\documentclass{article}
\usepackage{amsmath}
\usepackage[utf8]{inputenc}
\usepackage{graphicx}
\usepackage{stmaryrd}
\usepackage{amssymb}
\usepackage{xparse}
\usepackage{physics}
\usepackage{amsthm}
\usepackage{mathrsfs}
\usepackage{verbatim}
\usepackage{dsfont}
\usepackage{MnSymbol}
\usepackage{tikz}
\usepackage{qtree}
\usepackage{forest}
\usepackage{mathtools}
\usepackage{cancel}
\usepackage{hyperref}
\usepackage{MnSymbol}
\usepackage{blindtext}
\usepackage{algorithm}
\usepackage{algorithmic}
\usepackage{float}
\usepackage{multicol}

\usepackage{amsrefs}
\usepackage{ifthen}
\usepackage{hyperref}
\usepackage{todonotes}
\usepackage{datetime}
\DeclareFontFamily{U}{matha}{\hyphenchar\font45}
\DeclareFontShape{U}{matha}{m}{n}{
      <5> <6> <7> <8> <9> <10> gen * matha
      <10.95> matha10 <12> <14.4> <17.28> <20.74> <24.88> matha12
      }{}
\DeclareSymbolFont{matha}{U}{matha}{m}{n}
\DeclareFontSubstitution{U}{matha}{m}{n}

\DeclareMathSymbol{\precneq}{3}{matha}{"AC}
\newdateformat{monthyeardate}{\monthname[\THEMONTH] \THEYEAR}

\graphicspath{ {./images/} }
\newtheorem{theorem}{Theorem}
\newtheorem*{theorem*}{Theorem}
\newtheorem{lemma}{Lemma}
\newtheorem{cor}{Corollary}

\newtheorem{innercustomgeneric}{\customgenericname}
\providecommand{\customgenericname}{}
\newcommand{\newcustomtheorem}[2]{%
  \newenvironment{#1}[1]
  {%
   \renewcommand\customgenericname{#2}%
   \renewcommand\theinnercustomgeneric{##1}%
   \innercustomgeneric
  }
  {\endinnercustomgeneric}
}

\newcustomtheorem{customthm}{Theorem}

\theoremstyle{definition}
\newtheorem{defn}{Definition}
\theoremstyle{remark}

\newcommand{\defref}[1]{\hyperref[#1]{Definition~\ref*{#1}}}
\newcommand{\Defref}[1]{\hyperref[#1]{Definition~\ref*{#1}}}
\newcommand{\lemref}[1]{\hyperref[#1]{Lemma~\ref*{#1}}}
\newcommand{\Lemref}[1]{\hyperref[#1]{Lemma~\ref*{#1}}}
\newcommand{\thmref}[1]{\hyperref[#1]{Theorem~\ref*{#1}}}
\newcommand{\Thmref}[1]{\hyperref[#1]{Theorem~\ref*{#1}}}
\newcommand{\coref}[1]{\hyperref[#1]{Corollary~\ref*{#1}}}
\newcommand{\Corref}[1]{\hyperref[#1]{Corollary~\ref*{#1}}}
\newcommand{\figref}[1]{\hyperref[#1]{Figure~\ref*{#1}}}
\newcommand{\algref}[1]{\hyperref[#1]{Algorithm~\ref*{#1}}}

\tikzset{
          solid node/.style={circle,draw,inner sep=1.2,fill=black},
          green node/.style={circle,draw,inner sep=1.2,fill=green, draw=green},
          red node/.style={circle,draw,inner sep=1.2,fill=red, draw=red},
          blue node/.style={circle,draw,inner sep=1.2,fill=blue, draw=blue},
          yellow node/.style={circle,draw,inner sep=1.2,fill=yellow, draw=yellow},
          cyan node/.style={circle,draw,inner sep=1.2,fill=cyan, draw=cyan},
          magenta node/.style={circle,draw,inner sep=1.2,fill=magenta, draw=magenta},
          hollow node/.style={circle,draw,inner sep=1.2},
          big node/.style={elipse, draw=green!60, fill=green!5},
          big solid node/.style={circle,draw,inner sep=1.2,fill=black, minimum size=0.3cm},
          big green node/.style={circle,draw,inner sep=1.2,fill=green, draw=green, minimum size=0.3cm},
          big red node/.style={circle,draw,inner sep=1.2,fill=red, draw=red, minimum size=0.3cm},
          big blue node/.style={circle,draw,inner sep=1.2,fill=blue, draw=blue, minimum size=0.3cm},
          big yellow node/.style={circle,draw,inner sep=1.2,fill=yellow, draw=yellow, minimum size=0.3cm},
          big cyan node/.style={circle,draw,inner sep=1.2,fill=cyan, draw=cyan, minimum size=0.3cm},
          big magenta node/.style={circle,draw,inner sep=1.2,fill=magenta, draw=magenta, minimum size=0.3cm},  
          square blue node/.style={rectangle,draw,inner sep=1.2,fill=blue, draw=blue, minimum size=0.3cm},
          square red node/.style={rectangle,draw,inner sep=1.2,fill=red, draw=red, minimum size=0.3cm},
          square yellow node/.style={rectangle,draw,inner sep=1.2,fill=yellow, draw=yellow, minimum size=0.3cm},
          diamond blue node/.style={diamond,draw,inner sep=1.2,fill=blue, draw=blue, minimum size=0.3cm},
          diamond red node/.style={diamond,draw,inner sep=1.2,fill=red, draw=red, minimum size=0.3cm},
          diamond yellow node/.style={diamond,draw,inner sep=1.2,fill=yellow, draw=yellow, minimum size=0.3cm},
          left label/.style={above left,midway},
          right label/.style={above right,midway}
        }

\title{Cocompact unfolding trees}
\author{Roman Gorazd\footnote{University of Newcastle, email:\href{mailto:roman.gorazd@gmail.com}{roman.gorazd@gmail.com}}}
\date{\monthyeardate\today}

\tikzstyle{dot}=[fill=black, draw=black, shape=circle]
\tikzstyle{red_dot}=[fill=red, draw=red, shape=circle]
\tikzstyle{green_dot}=[fill=green, draw=green, shape=circle]
\tikzstyle{yellow_dot}=[fill=yellow, draw=yellow, shape=circle]
\tikzstyle{blue_dot}=[fill=blue, draw=blue, shape=circle]
\tikzstyle{orange_dot}=[fill=orange, draw=orange, shape=circle]

\tikzstyle{arrow}=[->]

\begin{document}

 \maketitle

\begin{abstract}
    This paper will show when a rooted path tree of a finite directed rooted graph has only finitely many orbits under the action of it's undirected automorphism group (i.e. when it is cocompact). This will allow us to specify which trees are almost isomorphic to cocompact trees. We will provide an algorithm that will determine this, thus mostly answering question (1) from \cite{COURCELLE2025}.
\end{abstract}
\section{Introduction}
In the study of groups acing on trees the class of cocompact trees (i.e. a tree where the (undirected) automorphism group has finitely many orbits) has been of particular interest, when looking at actions of groups on trees. A key example of groups acting on such trees is the Burger-Mozes universal group \cite{Burger-Mozes}, that in the cocompact case describes action on finite degree regular trees (trees with all vertices having the same degree), as limits of finite permutation groups. This concept was generalised in \cite{smith2017} to the box product describing actions on biregular trees. This was in turn generalised in \cite{Reid2020} to describe all actions on cocompact trees satisfying property (P)\cite{Tits1970}. Cocompact trees also originate as Bass-Serre covering trees of finite graphs of groups\cite{bass1993covering}\cite{serre2002trees}.

On the other hand, path spaces of finite graphs (which we will refer to as unfolding trees) are of particular interest in the study of Thompson-like groups, with key examples of such groups ''almost acting'' on an unfolding tree, being the Thompson groups\cite{Cannon1996ThompsonGroups}(when the trees are binary), the Higman-Thompson groups\cite{higman74}(where the trees are quasi-regular) and full groups of shifts of finite type\cite{matui2012topological}\cite{lederle2020topological}. These trees can be seen as universal covers of finite rooted directed graphs. I intend to connect these two types of trees by showing which rooted directed graphs have cocompact unfolding trees.  
This will use the results of \cite{Reid2020} and allow us to answer the question when a tree is almost isomorphic to a cocompact tree . For this, we first note that any tree that is almost isomorphic to a cocompact tree is an unfolding tree of a graph, since a cocompact tree is an unfolding tree (as we will show) and any tree that is almost isomorphic to an unfolding tree is also an unfolding tree. Now, suppose we can characterise the graphs with cocompact unfolding trees. In that case, we can use the results from \cite{gorazd2023classification} to determine the graphs whose unfolding trees are isomorphic to cocompact trees. To classify the graphs that unfold into a cocompact tree, we introduce a type of vertex labelling we call \textbf{actual} and show that any graph that admits such a labelling has a cocompact unfolding tree. Conversely, any sinkless graph with cocompact unfolding tree admits an actual labelling.

The content of this paper relates to \cite{COURCELLE2025} and mostly answers the open question (1) mentioned in its conclusion, by providing an algorithm that decides whether a finite graph unfolds into a cocompact tree (referred to as strongly regular in the article).

In this paper, the graphs will be directed with multiple edges. Formally, graphs are $4$-tuples $G=(VG,EG,o_G,t_G)$, where $VG$ and $EG$ are the vertex and edge sets respectively and $o_G,t_G:EG\to VG$ are the origin and terminus functions (we will suppress the subscripts if it is clear what graph we are referring to). We will refer to vertices $x$ s.t. $o^{-1}(x)=\emptyset$ as \textbf{sinks} and vertices $y$ s.t. $t^{-1}(y)=\emptyset$ as \textbf{sources}.
We can define the set of paths in $G$ to be
    \[\mathcal{W}(G):=\{e_1e_2\dots e_n\in EG^*\mid \forall i<n\ e_{i-1}=e_i\}\cup \{\varepsilon_v\mid v\in VG\}, \] 
where $\varepsilon_v$ denotes the empty path based on $v$. We will denote $\mathcal{W}^{+}(G)$ to be the set of non-empty paths. For any path $p\in \mathcal{W}(G)$, we denote $|p|$ to be the length of it, additionally we set $O(p),\ T(p)$ to be the origin vertex and terminal vertex of $p$, respectively (with $O(\varepsilon_v)=T(\varepsilon_v)=v$ for any vertex $v$). We will write $\mathcal{W}(G,v)$ to be the set of all paths originating in $v$. For any graph $G$ and any vertex $v$, we will denote by $G_v$ the subgraph consisting of the vertices and edges that can be reached by a directed path starting at $v$ (and all edges that start and end at those vertices). A vertex $R$ s.t. $G_R=G$, is called a root. A rooted graph will be a pair $(G,R)$ where $G$ is a graph and $R$ one of it's roots.
We call a graph a (rooted) \textbf{tree} if it has a root $R$ and for each vertex $v$ there is a unique path from $R$ to $v$ (note that this makes the root unique).

The following definitions will aim to translate the concept of a undirected unrooted tree being cocompact into the directed rooted setting which allows us to apply the results of \cite{gorazd2023classification}. 

\begin{defn}
    For any rooted tree $\mathcal{T}$ with the unique root $R$ and any $v\in V\mathcal{T}$ we can define the tree $\mathcal{T}^v$ by keeping the vertices and edges the same: $V\mathcal{T}^v:=V\mathcal{T}$ and $E\mathcal{T}^v:=E\mathcal{T}$, but reversing the edge directions of the edges on the unique path $p_v$ from $R$ to $v$ by setting:
        \[\forall e\in E\mathcal{T}^v,\ o_{\mathcal{T}^v}(e):=\begin{cases}
                t_{\mathcal{T}}(e), & e\in p_v\\
                o_{\mathcal{T}}(e), & e\notin p_v
            \end{cases}\quad t_{\mathcal{T}^v}(e):=\begin{cases}
                o_{\mathcal{T}}(e), & e\in p_v\\
                t_{\mathcal{T}}(e), & e\notin p_v
            \end{cases}\]
\end{defn}
This makes $\mathcal{T}^v$ the tree with root $v$ that is undirected isomorphic to $\mathcal{T}$ as unrooted trees. Note that if a homomorphism from $\mathcal{T}$ to $\mathcal{T}'$ maps a vertex $v\in \mathcal{T}$ to $w\in \mathcal{T}'$ then it induces a homomorphism from $\mathcal{T}^v$ to $(\mathcal{T}')^{w}$.
\begin{lemma}\label{ch2-lem-1}
    For any two rooted trees $\mathcal{T},\mathcal{T}'$, any homomorphism 
        \[\phi:\mathcal{T}\to\mathcal{T}',\]
    and any $v\in V\mathcal{T}$, there is a homomorphism
        \[\phi^v:\mathcal{T}^v\to(\mathcal{T}')^{\phi(v)},\]
    that agrees with $\phi$ on both vertices and edges.
\end{lemma}
\begin{proof}
    We will first note that $\phi$ takes the roots to each other.
    
    Now set $\phi^v(x)=\phi(x)$ for any $x\in E\mathcal{T}\cup V\mathcal{T}=E\mathcal{T}^v\cup V\mathcal{T}^v$. We will just have to show that this is in fact a homomorphism for this we take $p_v$ to be the path from the root to $v$, $p_{\phi(v)}$ to be the path from the root to $\phi(v)$ and some edge $e\in E\mathcal{T}$. Note that $\phi(p_v)=p_{\phi(v)}$.

    If $e$ is not in $p_v$, $\phi(e)=\phi^v(e)$ is also not in $p_{\phi(v)}$, so we have
        \begin{align*}
            o_{(\mathcal{T}')^{\phi(v)}}(\phi^v(e))&=o_{\mathcal{T}'}(\phi(e))=\phi(o_{\mathcal{T}}(e))=\phi(o_{\mathcal{T}^v}(e))\\
            t_{(\mathcal{T}')^{\phi(v)}}(\phi^v(e))&=t_{\mathcal{T}'}(\phi(e))=\phi(t_{\mathcal{T}}(e))=\phi(t_{\mathcal{T}^v}(e)).
        \end{align*}

    If the $e$ is in $p_v$, $\phi(e)=\phi^v(e)$ also is in $p_{\phi(v)}$, so we have
        \begin{align*}
            o_{(\mathcal{T}')^{\phi(v)}}(\phi^v(e))&=t_{\mathcal{T}'}(\phi(e))=\phi(t_{\mathcal{T}}(e))=\phi(o_{\mathcal{T}^v}(e))\\
            t_{(\mathcal{T}')^{\phi(v)}}(\phi^v(e))&=o_{\mathcal{T}'}(\phi(e))=\phi(o_{\mathcal{T}}(e))=\phi(t_{\mathcal{T}^v}(e)).
        \end{align*}
    This shows that $\phi^v$ is a homomorphism. 
\end{proof}
We can now define when a tree is \textbf{cocompact}. Note that, this definition is equivalent with saying that if we remove the directions of the edges and do not distinguish the root, then the automorphism group of the tree acts cocompactly on the set of vertices (endowed with the discrete topology).
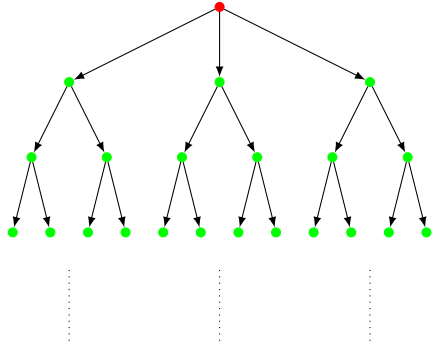
\begin{figure}
    \centering
    \begin{tikzpicture}
        [edge from parent/.style={draw,-latex},
            level distance=10mm,
            level 1/.style={sibling distance=20mm},
            level 2/.style={sibling distance=10mm,},
            level 3/.style={sibling distance=5mm,}
            ]
        \node[red node] {}
            child {node[green node] {}
                child {node[green node] {}
                    child {node[green node] {}}
                    child {node[green node] {}}
                }
                child {node[green node] {}
                    child {node[green node] {}}
                    child {node[green node] {}}
                }
            }
            child {node[green node] {}
                child {node[green node] {}
                    child {node[green node] {}}
                    child {node[green node] {}}
                }
                child {node[green node] {}
                    child {node[green node] {}}
                    child {node[green node] {}}
                }
            }
            child {node[green node] {}
                child {node[green node] {}
                    child {node[green node] {}}
                    child {node[green node] {}}
                }
                child {node[green node] {}
                    child {node[green node] {}}
                    child {node[green node] {}}
                }
            };
            \draw[dotted]   (-2,-3.5)--(-2,-4.5);
            \draw[dotted]   (0,-3.5)--(0,-4.5);
            \draw[dotted]   (2,-3.5)--(2,-4.5);
    \end{tikzpicture}
    \caption{This is a cocompact tree since if we change the root to be any green vertex we will obtain a tree that is isomorphic to the original tree.
    }
    \label{fig1}
\end{figure}
\begin{defn}
    A rooted tree $\mathcal{T}$ is cocompact (or undirected label-regular) if there exists a finite set $F$ and a labelling
        \[l:V\mathcal{T}\to F,\]
    s.t. for any $v,w\in V\mathcal{T}$
        \[l(v)=l(w)\implies \mathcal{T}^v\cong \mathcal{T}^w.\]
\end{defn}

To construct a tree from any rooted graph $(G,R)$ by looking at the set of paths
    \[\mathcal{W}(G)=\{e_1e_2\dots e_n\in EG^{<\infty}\mid \forall 1\leq i< n,\ t(e_i)=o(e_{i+1})\}\]
and enow it with the origin and terminus function $O(e_1e_2\dots e_n):=o(e_1)$ and $T(e_1e_2\dots e_n):=t(e_n)$. If we have two paths $p,q\in\mathcal{W}(G)$ we will call $p$ a prefix of $q$ of each other if we can write $q=pr$ for some $r\in\mathcal{W}(G)$. In this case, we will write $p\preceq q$. Note that we have $\varepsilon_v\preceq p$ if and only if $O(p)=v$, for any $v\in VG$.
We can construct the unfolding tree of the rooted graph $\mathcal{T}(G,R)$ by setting 
    \begin{align*}
        V\mathcal{T}(G,R):&=\{p\in\mathcal{W}(G)\mid O(p)=R\}\\
        E\mathcal{T}(G,R):&=\{(p,pe)\mid p\in V\mathcal{T}(G,R),\ e\in o^{-1}(T(p))\}
    \end{align*}
and of course $o(p,pe):=p$ and $t(p,pe):=pe$. If we change the root of an unfolding tree to some path $p$ we will change the name of the edge $(q,qe)$ to $(qe,q)$ in $\mathcal{T}(G,R)^p$, for any path $q\preceq p$ and any $e\in o^{-1}(T(q))$, s.t. $qe\preceq p$ as well, this reflects that in the new tree the origin and terminus of these edges are swapped .We can show when two graphs produce isomorphic unfolding trees by defining non-edge-collapsing(n.e.c.) graph equivalence relation
\begin{defn}
    A graph equivalence relation $\sim$ on the graph $G$ is a non-edge collapsing relation if for any two vertices $v,w\in VG$ there exists bijections 
        \[\pi_{v,w}:o^{-1}(v)\to o^{-1}(w)\]
    s.t. for any $e,f\in EG$
        \[e\sim f\iff \pi_{o(e),o(f)}(e)=f.\]
\end{defn}
Two graphs have the same unfolding trees if we can reduce them to the same graph via quotienting over a n.e.c. relation
\begin{lemma}{\cite{gorazd2023classification}*{Theorem 1, Lemma 4}}\label{lem-2}
    For any two rooted graphs $(G,R), (H,S)$ we have
        \[\mathcal{T}(G,R)\cong \mathcal{T}(H,S)\]
    if and only if there exists n.e.c. relations $\sim_G,\sim_H$ s.t.
        \[G/{\sim_G}\cong H/{\sim_H}\]
    via an isomorphism that takes $[R]_{\sim_G}$ to $[S]_{\sim_H}$.
\end{lemma}
For graphs $(G,R),(H,S)$ that are isomorphic after quotienting over n.e.c. relation we write 
    \[(G,R)\simeq(H,S).\]
A useful result for creating n.e.c equivalence relations is
\begin{lemma}{\cite{gorazd2023classification}*{Lemma 1}}\label{lem-1}
    For any graphs $G,H$, if we have a surjective graph morphism
        \[\phi:G\to H\]
    s.t. for each vertex $v\in VG$ $\phi|_{o^{-1}(v)}$ is a bijection $o^{-1}(v)\to o^{-1}(\phi(v))$, there is a n.e.c relation $\sim$ and an isomorphism
        \[\psi:G/{\sim}\to H\]
    s.t. $\forall x\in EG\cup VG,\ \psi([x]_{\sim})=\phi(x)$.
\end{lemma}

In order to show that when we change the root of an unfolding tree they still remain unfolding trees we will introduce a graph modification that causes it's unfolding tree to change roots. This procedure is demonstrated in \figref{fig2}. In order to do this we will take the path $p$ (the blue path in the left graph in \figref{fig2}) that we intend to be the new root of the tree and add vertices labelled $q$ for each prefix $q$ of $p$ (including the empty path and $p$ itself) to the graph. We will connect these vertices in reverse order, i.e. we will add edges from $qe$ to $q$ for any prefix $q$ of $p$ and any edge $e$ s.t. $qe\preceq p$ as well (this results in the blue path in the right graph in \figref{fig2}). Note that there is at most one such edge for any $q$. The added subgraph will correspond to the reversed path $p$ in the unfolding tree. To connect the reversed path to the original graph we look at each new vertex labelled with a prefix $q$ of $p$ and identify all the edges $d$ originating at $T(q)$, s.t when we expand $q$ by $d$ we exit the path $p$, i.e. $qd\not\preceq p$. For each of these prefixes and edges we add a corresponding edge $d_q$ going from the vertex labelled $q$ to $t(d)$ (the terminus of $d$ in the original graph). In \figref{fig2} those will be the orange edges. 
\begin{figure}
    \centering
    \begin{multicols}{2}
            \begin{tikzpicture}
                \node [big red node, label={above:$R$}]     (x) at (-2,-2)   {};
                \node[big green node, label={above:$v$}]   (y) at (2,-2)    {};

                \path[-latex]
                        (x)     edge[bend left,blue]                         node[above,black]             {$e$} (y)
                                edge[out=135, in=215, looseness=15]     node[left]              {$c$} (x)
                        (y)     edge[bend left]                         node[below]             {$d$} (x)
                               edge[out=45, in=315, looseness=15,blue]       node[right,black]             {$f$} (y)
                        ;    
            \end{tikzpicture}
        \columnbreak
        \begin{tikzpicture}
                \node[big red node, label={below:$R$}]     (x) at (-2,0)   {};
                \node[big green node, label={below:$v$}]   (y) at (2,0)    {};
                \node[big red node, label={above:$\varepsilon_R$}]     (-) at (-2,3)    {};
                \node[big green node, label={above:$e$}]   (e) at (0,3)    {};
                \node[big green node, label={above:$ef  $}]   (ef) at (2,3)   {};

                \path[-latex]
                        (x)     edge[bend left]                         node[above]             {$e$} (y)
                                edge[out=135, in=215, looseness=15]     node[left]              {$c$} (x)
                        (y)     edge[bend left]                         node[below]             {$d$} (x)
                                edge[out=45, in=315, looseness=15]      node[right]             {$f$} (y)
                        (ef)    edge[blue]                                  node[]                  {}  (e)
                        (e)     edge[blue]                                  node[]                  {}  (-)
                        (ef)    edge[orange]                                  node[right]             {$f_{ed}$}  (y)
                        (ef)    edge[orange]                                  node[right]             {$d_{ed}$}  (x)
                        (e)     edge[orange]                                  node[]             {$d_e$}  (x)
                        (-)     edge[orange]                                  node[left]             {$c_{\varepsilon_R}$} (x)
                        ;    
            \end{tikzpicture}
    \end{multicols}
    \caption{Graph $G$ on the left and graph $G^p$ for the blue path $p=ef$ on the right, note that the upper edges go from larger prefix to smaller prefix}
    \label{fig2}
\end{figure}
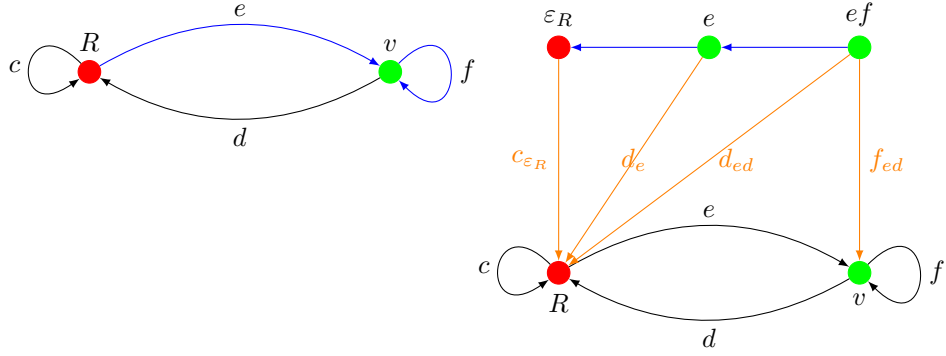

Formally we can describe the construction as follows.   
\begin{defn}\label{ch2-def-1}
    For any rooted graph $(G,R)$ and any path $p\in \mathcal{W}(G,R)$, define the graph $G^p$ to have the following extended vertex and edge set 
    \begin{align*}
        VG^p:=VG\cup &\{q\in \mathcal{W}(G,R)\mid q\preceq p\}\\
        EG^p:=EG\cup &\{(qe,q)\mid e\in EG,\ q,qe\preceq p\}\\
        \cup &\{d_q\mid q\preceq p,\ d\in E_o(T(q)),\ qd\not\preceq p \},
    \end{align*}
    with the $d_q$'s being new distinct edges, and define the origin and terminus functions for this graph ($o_{G^p}$ and $t_{G^p}$), by setting $o_{G^p}|_{EG}:=o_{G},\ t_{G^p}|_{EG}:=t_{G}$,
        \[\forall e\in EG,\ q,qe\preceq p,\ o_{G^p}(qe,q):=qe,\quad t_{G^p}(qe,q):=q\]
    and
        \[\forall q\preceq p,\ d\in E_o(T(q)),\ qd\not\preceq p,\ o_{G^p}(d_q):=q,\quad t_{G^p}(d_q):=t_G(d). \]
\end{defn}

Note that sometimes in this graph $p$ won't be a root, by convention we will then simply remove the vertices that are unreachable from $p$.
The intention of this construction is to allow us to construct a graph whose unfolding tree is $\mathcal{T}(G,R)$ with root $p$ instead of $\varepsilon_R$.\\
\begin{lemma}\label{ch2-lem0}
    For any $p\in\mathcal{T}(G,R)$ we have
        \[\mathcal{T}(G,R)^p\cong \mathcal{T}(G^p,p).\]
\end{lemma}
\begin{proof}
    We will define a surjective homomorphism that induces a non-edge-collapsing equivalence relation
        \[f:\mathcal{T}(G,R)^p\to G^p,\]
    as follows
        \[\forall q\in V\mathcal{T}(G,R),\ f_V(q):=   \begin{cases}
                                                        q, & q\preceq p\\
                                                        T(q), & q\not\preceq p
                                                    \end{cases} \]
        and 
        \[\forall q\in V\mathcal{T}(G,R),\ e\in E_o(T(p)),\ f_E(q,qe):=                                                 \begin{cases}
                                                        (qe,q), & q,qe\preceq p\\
                                                        e_q, & q\preceq p,\ qe\not\preceq p\\
                                                        e, & q\not\preceq p
                                                    \end{cases}. \]
    To see that this is in fact a homomorphism we have to look at some edge $(q,qe)\in E\mathcal{T}(G,R)$, and distinguish the  following cases.
    
    If $q,qe\preceq p$ we have
        \[o_{G^p}(f(q,qe))=o_{G^p}((qe,q))=qe=f(qe)=f(o_{\mathcal{T}(G,R)^p}(q,qe))\]
    and 
        \[t_{G^p}(f(q,qe))=t_{G^p}((qe,q))=q=f(q)=f(t_{\mathcal{T}(G,R)^p}(q,qe)),\]
    since $(q,qe)$ is on the path between $\varepsilon_R$ to $p$ in $\mathcal{T}(G,R)$.

    If $q\preceq p$, but $qe\not\preceq p$ we have
        \[o_{G^p}(f(q,qe))=o_{G^p}(e_q)=q=f(q)=f(o_{\mathcal{T}(G,R)^p}(q,qe))\]
    and
        \[t_{G^p}(f(q,qe))=t_{G^p}(e_q)=t_{G}(e)=f(qe)=f(t_{\mathcal{T}(G,R)^p}(q,qe)),\]
    since $(q,qe)$ is not on the path between $\varepsilon_R$ to $p$ in $\mathcal{T}(G,R)$.

    If $q\not\preceq p$ and $qe\not\preceq p$ we have 
        \[o_{G^p}(f(q,qe))=o_{G^p}(e)=o_G(e)=T(q)=f(q)=f(o_{\mathcal{T}(G,R)^p}(q,qe))\]
    and
        \[t_{G^p}(f(q,qe))=t_{G^p}(e)=o_G(e)=T(q)=f(q)=f(o_{\mathcal{T}(G,R)^p}(q,qe)),\]
    since again $(q,qe)$ is not on the path between $\varepsilon_R$ to $p$ in $\mathcal{T}(G,R)$.

    Now to show that $f$ induces a non-edge-collapsing equivalence relation, we look at each $p\in \mathcal{T}(G,R)$ and distinguish the following cases.

    If $q\not\preceq p$ we have
        \[o_{\mathcal{T}(G,R)^p}^{-1}(q)=\{(q,qe)\mid e\in o^{-1}(T(p))\}\]
    in $\mathcal{T}(G,R)^p$. Since $f(q,qe)=e$ for each $e\in o^{-1}(T(q))=o_{G^p}^{-1}(T(q))$, we can see that $f|_{o^{-1}(q)}$ is a bijection.

    If $\varepsilon_R\neq q \precneq p$, take $e_+\in EG$ to be the unique edge s.t $qe_+\preceq p$, and take $q_{-}\preceq q$ and $e_-\in o^{-1}(T(q_{-}))$ to be s.t. $q_-e_-=q\preceq p$ we have
        \[o_{\mathcal{T}(G,R)^p}^{-1}(q)=\{(q,qe)\mid e\in E_o(T(p))\setminus\{e_+\}\}\cup\{(q_{-},q)\}\]
    and
        \[o_{G^p}^{-1}(q)=\{(q,e)\mid  e\in o^{-1}(T(p))\setminus\{e_+\}\}\cup \{(q,q_-)\}.\]
    From this form of the outgoing neighbourhoods we can see that $f|_{o_{\mathcal{T}(G,R)^p}^{-1}(q)}$ is a bijection.

    If $q=\varepsilon_R$ we take $e_+\in EG$ as before ($q_-$ and $e_-$ cannot be defined). Then
        \[o_{\mathcal{T}(G,R)^p}^{-1}(q)=\{(q,qe)\mid e\in o^{-1}(T(p))\setminus\{e_+\}\}\]
    and
        \[o_{G^p}^{-1}(q)=\{(q,e)\mid  e\in o^{-1}(T(p))\setminus\{e_+\}\}.\]
    From this form of the outgoing neighbourhoods we can see that $f|_{o_{\mathcal{T}(G,R)^p}^{-1}(q)}$ is a bijection.

    If $q=p$ we take $q_-$ and $e_-$ as before ($e_+$ doesn't exist). Then
        \[o_{\mathcal{T}(G,R)^p}^{-1}(q)=\{(q,qe)\mid e\in o^{-1}(T(p))\}\cup \{(q,q_-)\}\]
    and
        \[o_{G^p}^{-1}(q)=\{(q,e)\mid  e\in o^{-1}(T(p))\}\cup\{(q_-,q)\}.\]
    From this form of the outgoing neighbourhoods we can see that $f|_{o_{\mathcal{T}(G,R)^p}^{-1}(q)}$ is a bijection.

    So, by \lemref{lem-1}, there is some non-edge-collapsing equivalence relation $\sim$ on $\mathcal{T}(G,R)^p$ s.t.
        \[\mathcal{T}(G,R)^p/{\sim}\cong G^p,\]
    and thus
        \[\mathcal{T}(G,R)^p\cong\mathcal{T}(\mathcal{T}(G,R)^p,p)\cong\mathcal{T}(G^p,p).\]
\end{proof}

This result allows us to restate the property of an unfolding tree being cocompact as saying that
    \[|\{(G^p,p)\mid p\in\mathcal{T}(G,R)\}/{\simeq}|<\infty.\]
\section{Actual labelling}
To better describe the neighbourhoods of a vertex, we will use multisets:
\begin{defn}
    A multiset is a tuple $\mathbf{M}=(M,m)$ where $M$ is the base set and $m:M\to \mathds{Z}_{>0}$ is the multiplicity function. We will treat an ordinary set $N$ as multisets with base set $N$ and  all multiplicities being $1$.
\end{defn}
    For an multisets $M,N$, we will write $M\subseteq_b N$ to indicate that $M$ is a multiset with base set contained in the base set of $N$. We write $M\subseteq N$ if $M\subseteq_b N$ and for each $x\in M$ $m_M(x)\leq m_N(x)$. We define the cardinality $|\cdot|$ of the multiset as the sum of the multiplicities.
    If we have a function $f$ and a multiset $X$ with base set in the domain we can define the multi-image of the set $f(X)$, by having the base set being $f(X)$ and the multiplicity of each $f(x)$ being the sum of the multiplicities of $y$'s s.t. $f(y)=f(x)$. When we want to treat a ordinary set $X$ as a multiset in an image we will write $f_m(X)$.  The union of multisets is defined by taking the union of the base sets and adding the multiplicities. The set difference of two sets is defined by subtracting the multiplicities of one set from the other and discarding elements with multiplicity $\leq 0$.

To show that a cocompact tree $\mathcal{T}$ with root $R$ is an unfolding tree of some finite graph, we can start by labelling each vertex $q$ of a tree $\mathcal{T}$ with the isomorphism class of $\mathcal{T}^q$, i.e. $l(q)=[\mathcal{T}^q]_{\cong}$. This labelling however does not induce a non-edge-collapsing equivalence relation in general. To achieve a labelling that can be extended in this way we will have to look at the predecessor $q_{-}$ of $q$ i.e. the unique vertex that has $q$ as an outgoing neighbour and set $l_p(q)=[\mathcal{T}^{q_{-}}]_{\cong}$. Note that if we have $l(q_1)=l(q_2)$ and $l_p(q_1)=l_p(q_2)$ for some $q_1,q_2\in V\mathcal{T}$ that are not the root, we must have $\mathcal{T}_{q_1}\cong \mathcal{T}_{q_2}$. This shows that there is a non-edge collapsing equivalence relation $\sim$ on $\mathcal{T}$ s.t.
    \[q_1\sim q_2\iff (l(q_1),l_p(q_1))=(l(q_2),l_p(q_2)),\]
for any vertices that are not the root. Since $\mathcal{T}$ is cocompact, $\mathcal{T}/{\sim}$ is finite. Of course the labels $l$ and $l_p$ filter through $\sim$ and we can apply them to vertices of $G$ as well. We will look at this graph as the canonical graph that unfolds into $\mathcal{T}$. We will later show that it is in fact the smallest in the sinkless case.
\begin{defn}\label{ch2-def0}
    For any cocompact rooted tree $(\mathcal{T},R)$, the canonical graph is defined as
        \[CG(\mathcal{T},R)=\mathcal{T}/{\sim},\]
    with $\sim$ being the non-edge collapsing equivalence defined above. 
\end{defn}

The labels $l,l_p$ have certain useful properties that make them into a labelling we will call actual, that is defined as follows.
\begin{defn}\label{ch2-def1}
    For any rooted graph $(G,R)$, with $t^{-1}(R)=\emptyset$ and any finite set $X$ the pair of labellings $l:VG\to X$ and $l_p:VG\setminus \{R\}\to X$ is \textbf{actual} if for each $x\in X$ there is a multiset $M_x\subseteq_b X$ s.t. for each $v\in VG\setminus \{R\}$
    \begin{enumerate}
        \item $l_p(v)\in M_{l(v)}$
        \item $l(t_m(o^{-1}(v)))= M_{l(v)}\setminus\{l_p(v)\}$
        \item $\forall e\in t^{-1}(\{v\}),\ l(o(e))=l_p(v)$
        \item $l(t_m(o^{-1}(R)))= M_{l(R)}$
    \end{enumerate}
    We will call a pair labellings on an unrooted graph that satisfies the first 3 of the above points \textbf{pre-actual} .
\end{defn}
We will show in this paper that if a graph admits an actual labelling its unfolding tree is cocompact and conversely any sinkless graph with cocompact unfolding tree has an actual labelling. The latter is relatively straightforward, while the former requires some more work.

Note that if we take the labelling $(l_p,l)$, it will induce a non-edge-collapsing equivalence relation.
\begin{lemma}\label{ch2-lem0.5}
    The relation $\sim$ on $VG$ defined as:
        \[v\sim w \iff (l_p(v),l(v))=(l_p(w),l(w))\]
    for each $v,w\in VG$, can be extended onto edges to form a non-edge-collapsing equivalence relation, $\sim_l$.
\end{lemma}
\begin{proof}
    Define for $v\in V$ the labelling $l_a(v)=(l_p(v),l(v))$.\\
    We will fix for each $v,w\in VG$, with $l_a(v)=l_a(w)$ a bijection $\pi_{v,w}:o^{-1}(v)\to o^{-1}(w)$ s.t. $l(t(e))=l(t(\pi_{v,w}(e)))$, which exists since $l(t_m(o^{-1}(v)))=l(t_m(o^{-1}(w)))$, by the second property of \defref{ch2-def1}. We can choose these bijections s.t. $\pi_{v,v}=\text{id}_{o^{-1}(v)}$ for any vertex $v$ and for any $3$ vertices $u,v,w$: $\pi_{u,v}\circ \pi_{v,w}=\pi_{u,w}$.

    Then by setting
        \[\forall v,w\in VG,\ v\sim w\iff l_a(v)=l_a(w)\]
    and,
        \[\forall e,f\in EG,\ e\sim f\iff (o(e)\sim o(f))\land \pi_{o(e),o(f)}(e)=f\]
    this gives us a graph equivalence relation that is non-edge collapsing.
\end{proof}
Note that after factoring over this equivalence relation we still get a graph with an actual labelling just by setting
    \[l([v]_{\sim_l})=l(v)\quad\text{and}\quad l_p([v]_{\sim_l})=l_p(v)\]
for any $v\in VG$.

One can show that an actual pair of labellings exists if and only if the unfolding tree of a graph is cocompact. To show this we will first show that:
\begin{lemma}\label{ch2-lem14}
    For any rooted tree $\mathcal{T}$ with an actual labelling, for each $p,q\in \mathcal{T}$ with $l(p)=l(q)$ we have
        \[\mathcal{T}^{p}\cong\mathcal{T}^{q}.\]
\end{lemma}
\begin{proof}
    Note that if $\mathcal{T}$ has an actual labelling $(l,l_p)$ then $\mathcal{T}^{p}$ also has an actual labelling $(l,\tilde{l}_p)$, with $\tilde{l}_p(v)$ being the label of the immediate predecessor of any vertex $v$ in $\mathcal{T}^p$. We may thus assume that $p$ is the root of $\mathcal{T}$.

    Note that for $\sim$ as in \lemref{ch2-lem0.5} the graph $G=\mathcal{T}/{\sim}$ is finite and $\mathcal{T}\cong\mathcal{T}(G,R)$ where $R$ is the equivalence class of the root of $\mathcal{T}$. So we can assume $\mathcal{T}=\mathcal{T}(G,R)$ going forwards and so by previous assumption $p=\varepsilon_R$.

    We can of course also assume that $q\neq \varepsilon_R=p$ and write $q=d_1d_2\dots d_N$ for $N\geq 1$ and set $q_n:=d_1d_2\dots d_n$ for any $n\leq N$.
    We will recursively construct a map
        \[\phi:\mathcal{T}(G, R)\to\mathcal{T}(G, R)^{q}\]
    such that $\phi(\varepsilon_R)=q$ and $\forall r\in \mathcal{T}(G, R),\ l(\phi(r))=l(r)$.
    
    Firstly, note that since $l(\varepsilon_R)=l(q)$, we have some $e_0\in o_{\mathcal{T}(G,R)}^{-1}(\varepsilon_R)$ s.t. there exist a bijection $\pi:o_{\mathcal{T}(G,R)}^{-1}(\varepsilon_R)\setminus\{e_0\}\to o^{-1}(q)$, that preserves the labels and $l(t(e_0))=l_p(q)=l(q_{N-1})$. We will thus define $\phi(e)=\pi(e)$ for any $e\in o^{-1}(\varepsilon_R)\setminus\{e_0\}$ and $\phi(e_0)=(q_{N-1},q_N)$ (the edge from $q_N$ to $q_{N-1}$).

    Now if we have some non-empty $r\in\mathcal{T}(G,R)$, with $\phi$ already defined on $r$ and all of its prefixes, we will define $\phi$ on $o^{-1}(r)$ as follows.Take $r_{-}$ to be $r$ with its last edge removed.
    
    If $\phi(r)$ is not a prefix of $q$ then $\phi(r_{-})$ is also the immediate prefix of $\phi(r)$ and thus
        \[l_p(r)=l(r_{-})=l(\phi(r_{-}))=l_p(\phi(r)).\]
    So we have a bijection $\pi_r:o^{-1}(r)\to o^{-1}(\phi(r))$. Now we can define $\phi(e)=r\pi_r(e)$ for any $e\in o^{-1}(r)$.

    If $\phi(r)$ is a proper prefix of $q$, then there is a $k<N$ s.t. $\phi(r)=q_k$ and $\phi(r_{-})=q_{k+1}$. So since by induction
        \[l_p(r)=l(r_{-})=l(q_{k+1}),\]
    we have, by properties of the actual labelling, an edge $e_0\in o^{-1}(r)$ with $l(t(e_0))=l_p(q_k)=l(q_{k-1})$ and a label respecting bijection $\pi_r:o^{-1}(r)\setminus\{e_0\}\to o^{-1}(q_k)\setminus\{(q_k,q_{k+1})\}$ now we can define for any $e_0\neq e \in o^{-1}(r)$, $\phi(e):=\pi_r(e)$ and $\phi(e_0)=(q_{k-1},q_{k})$.

    This defines a tree isomorphism between $\mathcal{T}(G,R)$ and $\mathcal{T}(G,R)^{q}$.
\end{proof}

This result shows that any tree with an actual labelling is cocompact.

Conversely, we can see that each actually labelled graph unfolds into another actually labelled tree by showing the following.
\begin{lemma}\label{ch2-lem0.6}
   Any rooted graph $(G,R)$, with a non-edge-collapsing equivalence relation $\sim$, s.t. $G/{\sim}$ has an actual labelling,  has an actual labelling as well.
\end{lemma}
\begin{proof}
    Note first that since we must have $t^{-1}([R]_{\sim})=\emptyset$, we also must have $t^{-1}(R)=\emptyset$. Now set for each $v\in VG$ $l(v):=l([v]_{\sim})$, $l_p(v):=l_p([v]_{\sim})$ and $M_{l(v)}=M_{l([v]_{\sim})}$. To show that this is indeed an actual labelling, we check all the conditions from \defref{ch2-def1}. Clearly $l_p(v)=l_p([v]_{\sim})\in M_{l([v]_{\sim})}=M_{l(v)}$, giving us the first point. 

    Since $\sim$ is a non-edge-collapsing equivalence relation, we have a bijection $\pi:o^{-1}(v)\to o^{-1}([v]_{\sim})$ sending $e\in o^{-1}(v)$ to $[e]_{\sim}$ and thus $t(\pi(e))=[t(e)]_{\sim}$, so $\pi$ induces a bijection $\pi:t_m(o^{-1}(v))\to t_m(o^{-1}([v]_{\sim}))$ that preserves the labels. Therefore, we have for any $v\neq R$
        \[l(t_m(o^{-1}(v)))=l(t_m(o^{-1}([v]_{\sim})))=M_{l([v]_{\sim})}\setminus\{l_p([v]_{\sim})\}=M_{l(v)}\setminus\{l_p(v)\}.\]
    The fourth point of the definition follows the same way since
        \[l(t_m(o^{-1}(R)))=l(t_m(o^{-1}([R]_{\sim})))=M_{l([R]_{\sim})}=M_{l(R)},\]

    Finally, for the third point, we will simply use that $\sim$ is a graph equivalence relation, and thus for every edge $e$ we have 
        \[l(o(e))=l([o(e)]_\sim)=l(o([e]_{\sim}))=l_p(t([e]_{\sim}))=l_p([t(e)]_{\sim})=l_p(t(e)),\]
    showing the third point.
\end{proof}

\section{Labelling graphs that unfold into cocompact trees }

Finally, we can ask ourselves whether, if the unfolding tree of a graph is cocompact the graph admits an actual labelling. Generally, this is not the case, as there are graphs with sinks that have cocompact unfolding trees but do not have an actual labelling. However, if we exclude graphs with sinks, we can show this. Note that this only excludes the trees that have sinks. From now on we will assume that the graphs that we are working with do not have sinks.
For this, we note the following.

\begin{lemma}\label{ch2-lem1}
    For any vertex $v\in VG$ and any two non-empty paths $p,q\in\mathcal{T}(G,R)$ ending in the same vertex $v$, if we have 
        \[\mathcal{T}(G,R)^p\cong\mathcal{T}(G,R)^q,\]
    then there exists an isomorphism
        \[\phi:\mathcal{T}(G,R)^p\to\mathcal{T}(G,R)^q,\] 
     s.t. \(\forall r\in\mathcal{T}(G,v),\ \phi(pr)=qr\).    
\end{lemma}
In the proof of this lemma, we will view  $\mathcal{T}(G,R)_p$ (the subtree of $\mathcal{T}(G,R)$ consisting of all paths that have $p$ as a prefix), to be also the subtree of $\mathcal{T}(G,R)^p$ that is induced by the vertices $V\mathcal{T}(G,R)_p\subseteq V\mathcal{T}(G,R)=V\mathcal{T}(G,R)^p$.
\begin{proof}
    First we write $o^{-1}(v)=\{e_1,\dots,e_n\}$ and take $p_{-}\preceq p$ and $q_{-}\preceq q$ s.t. $p=p_{-}e_p$ and $q=q_{-}e_q$ for some $e_p,e_q\in EG$. In $\mathcal{T}(G,R)^p$ the neighbours of $p$ are $p_{-},pe_1,\dots,pe_n$ and in $\mathcal{T}(G,R)^q$ the neighbours of $q$ are $q_{-},qe_1,\dots,qe_n$.\\
    Now take some isomorphism
        \[\theta:\mathcal{T}(G,R)^p\to\mathcal{T}(G,R)^q \]
    and we will distinguish 3 cases.

    If $\theta(p_{-})=q_{-}$ we must have $\theta(\mathcal{T}(G,R)_p)=\mathcal{T}(G,R)_q$ since \(\mathcal{T}(G,R)_p\) is the complement of \(\mathcal{T}(G,R)^p_{p_{-}}\) and \(\mathcal{T}(G,R)_q\) is the complement of \(\mathcal{T}(G,R)^q_{q_{-}}\).

    If $\theta(p_{-})=qe_i$ and $\theta(pe_{i})=q_{-}$ for some, $1\leq i\leq n$ then we have:
        \[\mathcal{T}(G,R)^q_{q_{-}}\cong \mathcal{T}(G,R)^p_{pe_i}\cong \mathcal{T}(G,R)_{pe_i}\cong \mathcal{T}(G,R)_{qe_i}\cong \mathcal{T}(G,R)^q_{qe_i} \cong \mathcal{T}(G,R)^p_{p_{-}}\]
    So we can define an isomorphism \(\phi:\mathcal{T}(G,R)^p\to\mathcal{T}(G,R)^q\) that takes $\mathcal{T}(G,R)^p_{p_{-}}$ to $\mathcal{T}(G,R)^q_{q_{-}}$ via the above isomorphism and takes $\mathcal{T}(G,R)^p_{pe_i}$ to $\mathcal{T}(G,R)^q_{qe_i}$ via the standard isomorphism. This is the required isomorphism.

    If $\theta(p_{-})=qe_i$ and $\theta(pe_{j})=q_{-}$ for some $i\neq j$. Define $k_i$ to be the number of half trees in $\{\mathcal{T}(G,R)_{pe_1},\dots, \mathcal{T}(G,R)_{pe_n}\}$ that are isomorphic to $\mathcal{T}(G,R)_{pe_i}$. Note that this is the same as the number of half trees in $\{\mathcal{T}(G,R)_{qe_1},\dots, \mathcal{T}(G,R)_{qe_n}\}$ that are isomorphic to $\mathcal{T}(G,R)_{pe_i}$ (as $T(p)=T(q)$). If we in turn look at the number of half-trees in $\{\mathcal{T}(G,R)_{pe_1},\dots, \mathcal{T}(G,R)_{pe_n}, \mathcal{T}(G,R)^p_{p_{-}}\}$ that are isomorphic to, $\mathcal{T}(G,R)_{pe_i}$ it will be $k_i+1$ since the restriction of $\theta$ provides an isomorphism $\mathcal{T}(G,R)^p_{p_{-}}\cong \mathcal{T}(G,R)_{pe_i}$. This, in turn, must also be the number of half-trees in $\{\mathcal{T}(G,R)_{qe_1},\dots, \mathcal{T}(G,R)_{qe_n}, \mathcal{T}(G,R)^q_{q_{-}}\}$ that are isomorphic to $\mathcal{T}(G,R)_{pe_i}$ (as $\theta$ maps these set to each other). So we must have $\mathcal{T}(G,R)^q_{q_{-}}\cong \mathcal{T}(G,R)_{pe_i}\cong \mathcal{T}(G,R)^p_{p_{-}}$. Now we can construct the required isomorphism as in the previous case.

    So we always can assume that $\theta(\mathcal{T}(G,R)_p)=\mathcal{T}(G,R)_q)$. By rearranging, we may assume that $\theta|_{\mathcal{T}(G,R)_p}$ takes $pr$ to $qr$ for each $r\in\mathcal{T}(G,v)$.
\end{proof}

\begin{lemma}\label{ch2-lem2}
    For any path $p,r$ in $G$, s.t. $O(p)=R$ and $T(p)=O(r)=T(r)$. Then, if 
        \[\mathcal{T}(G,R)^{pr}\cong \mathcal{T}(G,R)^{pr^2},\]
    we also have
        \[\mathcal{T}(G,R)^{p}\cong \mathcal{T}(G,R)^{pr}.\]
\end{lemma}
\begin{proof}
    Write $r=e_1e_2\dots e_n$, for edges $e_i$, and define $r_i:=e_1\dots e_i$ for any $1\leq i \leq n$, we will prove inductively that 
        \[\mathcal{T}(G,R)^{pr_{n-j}}\cong \mathcal{T}(G,R)^{prr_{n-j}},\]
    for any $0\leq j \leq n$.\\
    For $j=0$ the proposition follows from the assumption.

    Inductively, if the proposition is true for $j<n$, we have by \lemref{ch2-lem1} a isomorphism
        \[\theta:\mathcal{T}(G,R)^{pr_{n-j}}\to \mathcal{T}(G,R)^{prr_{n-j}},\]
    that takes $\mathcal{T}(G,R)_{pr_{n-j}}$ to $\mathcal{T}(G,R)_{prr_{n-j}}$ and thus must take $pr_{n-j-1}$ to $prr_{n-j-1}$. Therefore, we can reinterpret $\theta$ as an isomorphism from $\mathcal{T}(G,R)^{pr_{n-j-1}}$ to $\mathcal{T}(G,R)^{prr_{n-j-1}}$.

    So if we look at the proposition at $j=n$, we get the lemma.
\end{proof}
    Now  if we have some rooted graph $(G,R)$ and we know that their unfolding tree is cocompact we can consider the finite set 
        \[X:=\{\mathcal{T}(G,R)^p\mid p\in\mathcal{T}(G,R) \}/{\cong}\]
    and use its finite subsets to label the graph as follows
        \[\forall v\in VG,\ l(v)=\{[\mathcal{T}(G,R)^p]_{\cong}\mid T(p)=v \}.\]
    To show that this labelling can be extended to an actual labelling we will restrict at first to strongly connected graphs, i.e. graphs where for each $v,w\in VG$ there is a path $p$ in $G$ from $v$ to $w$.
    In this case, we can show that this labelling satisfies the third condition from \defref{ch2-def1}(away from the root), i.e. the label of all predecessors of a vertex is the same.
\begin{lemma}\label{ch2-lem3}
    Let $(G,R)$ be a rooted graph s.t. $t^{-1}(R)=\emptyset$, $G\setminus\{R\}$ is strongly connected and $\mathcal{T}(G,R)$ is cocompact, then
        \[\forall e_1,e_2\in t^{-1}(v)\setminus o^{-1}(R),\ l(o(e_1))=l(o(e_2)). \]
\end{lemma}
\begin{proof}
    Take $e_1,e_2$ to be edges that end in $v$ and $w_1,w_2$ to be their respective origins. It suffices to show that for any path $\tilde{p}$ from $R$ to $w_1$,  there is a path $q$ from $R$ to $w_2$ s.t.
        \[\mathcal{T}(G,R)^{\tilde{p}}\cong\mathcal{T}(G,R)^q.\]
    To construct such a path, we will extend $\tilde{p}$ by $e_1$ to get $p:=\tilde{p}e_1$. Additionally, we take a path $\tilde{r}$ from $v$ to $w_2$ and extend it by $e_2$ to get $r:=\tilde{r}e_2$. Now if we look at the infinite collection of paths $pr^n$, we see that since $\mathcal{T}(G,R)$ is cocompact we must have some $m,n>0$ s.t.
        \[\mathcal{T}(G,R)^{pr^n}\cong\mathcal{T}(G,R)^{pr^{n+m}},\]
    but by applying \lemref{ch2-lem2} $n$ times we get
        \[\mathcal{T}(G,R)^{p}\cong\mathcal{T}(G,R)^{pr^m}.\]
    Furthermore since $p$ and $pr^n$, have the same endpoint we can assume by \lemref{ch2-lem1} that the isomorphism takes $\tilde{p}$ to $pr^{n-1}\tilde{r}$. We get the desired result by setting $q=pr^{n-1}\tilde{r}$.
\end{proof}

We will now denote for any vertex $v\in VG\setminus\{R\}$ by $l_p(v)$ the unique label of the predecessors of $v$ (that are not a root).

Additionally, we can show the following using the proof of \lemref{ch2-lem3}.
\begin{cor}\label{ch2-cor0}
    For any two edges $e_1,e_2$ with $t(e_1)=t(e_2)$ and $o(e_1)=R$ we have a path $q\in\mathcal{T}(G,R)$ with $T(q)=o(e_2)$, s.t.
        \[\mathcal{T}(G,R)\cong\mathcal{T}(G,R)^{q}.\]
\end{cor}
\begin{proof}
    We apply the proof of \lemref{ch2-lem3} with $\tilde{p}=\varepsilon_R$.
\end{proof}

If we can define for each vertex $v\in VG\setminus \{R\}$ the multiset
    \[M_v:=l_m(t_m(o^{-1}(v)))\cup\{l_p(v)\},\]  
which is the multiset consisting of the labels of each outgoing neighbour of $v$ together with $l_p(v)$. The multiplicity of each label $\neq l_p(v)$ being the amount of edges pointing from $v$ to vertices of such label. For $l_p(v)$ we add $1$ to this multiplicity. 
We will show that these multisets only depend on the label of $v$, when the graph is non-redundant.

\begin{lemma}\label{ch2-lem5}
    For any non-redundant rooted graph $(G,R)$ s.t. $t^{-1}(R)=\emptyset$ and the graph $G\setminus\{R\}$ is strongly connected, we have for each $v,w\in VG\setminus\{R\}$ with $l(v)=l(w)$ we have
        \[M_v=M_w.\]
\end{lemma}
\begin{proof}

    We will assume that $v\neq w$ and thus $\mathcal{T}(G,R)_v\not\cong \mathcal{T}(G,R)_w$. Since $v,w$ have the same labels, if we have some $p\in \mathcal{T}(G,R)$ that ends in $v$ then we have some $q\in \mathcal{T}(G,R)$, with $T(q)=w$ s.t. an isomorphism
        \[\phi:\mathcal{T}(G,R)^{p}\to\mathcal{T}(G,R)^{q}\]
    exists.
    
    Note that we must have $\phi(p_{-})\neq q_{-}$, since if $\phi(p_{-})=q_{-}$, $\phi$ would take $\mathcal{T}(G,R)_v$ to $\mathcal{T}(G,R)_w$ as in the proof of \lemref{ch2-lem1}. This allows us to take  edges $e_0\in o^{-1}(v)$ and $d_0\in o^{-1}(w)$, s.t. $\phi(p_{-})=qd_0$ and $\phi(pe_0)=q_{-}$. Now for all the other edges that originate in $v$ we have a bijection 
        \[\pi:o^{-1}(v)\setminus\{e_0\}\to o^{-1}(w)\setminus\{d_0\} \]
    s.t. $\phi(pe)=q\pi(e)$ for each $e\in o^{-1}(v)\setminus\{e_0\}$ and thus we have 
        \[\mathcal{T}(G,R)_{t(e)}\cong \mathcal{T}(G,R)_{t(\pi(e))}.\]
    and additionally
         \[ \mathcal{T}(G,R)^{p}_{p_{-}}\cong \mathcal{T}(G,R)_{t(d_0)}.\]
    
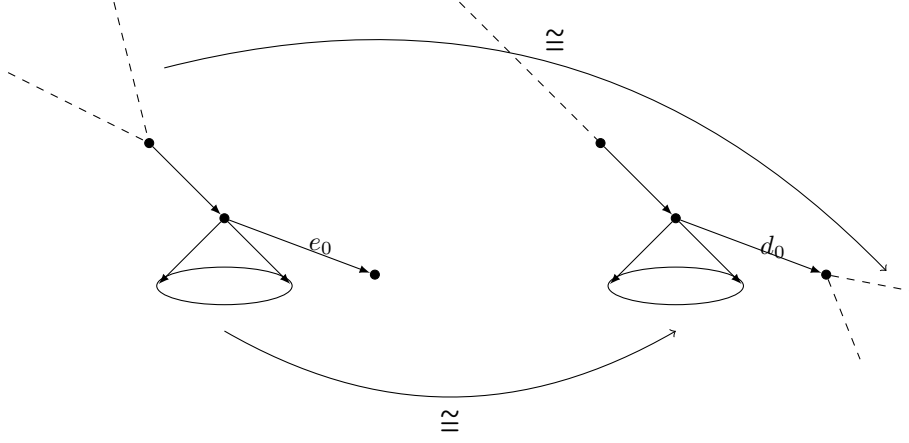
\begin{center}
            \begin{tikzpicture}
                \node[]                 (p1--) at (-7, 2)   {};
                \node[]                 (p2--) at (-5.5, 3)   {};
                \node[solid node]       (p-) at (-5,1)      {};
                \node[solid node]       (p) at (-4,0)       {};
                \node[solid node]       (p0) at (-2,-0.75)  {};
                \node[]                 (p1) at (-3,-1)     {};
                \node[]                 (p2) at (-5,-1)     {};
                \node[]                 (q--) at (-1, 3)     {};
                \node[solid node]       (q-) at (1,1)       {};
                \node[solid node]       (q) at (2,0)        {};
                \node[solid node]       (q0) at (4,-0.75)   {};
                \node[]                 (q01) at (4.5,-2)     {};
                \node[]                 (q02) at (5.3,-1)     {};
                \node[]                 (q1) at (1,-1)      {};
                \node[]                 (q2) at (3,-1)      {};
                \path[-latex]
                        (p-)        edge                                 node[above]           {}       (p)
                        (p)         edge                                 node[right]           {$e_0$}  (p0)
                                    edge                                 node[right]           {}       (p1)
                                    edge                                 node[right]           {}       (p2)

                        (q-)        edge                                 node[above]           {}       (q)
                        (q)         edge                                 node[right]           {$d_0$}  (q0)
                                    edge                                 node[right]           {}       (q1)
                                    edge                                 node[right]           {}       (q2)

                          ;
                \draw (-4,-0.9) ellipse (0.9 and 0.25); 
                \draw (2,-0.9) ellipse (0.9 and 0.25); 
                \draw[dashed] (p1--) -- (p-);
                \draw[dashed] (p2--) -- (p-);
                \draw[dashed] (q--) -- (q-);
                \draw[dashed] (q0) -- (q01);
                \draw[dashed] (q0) -- (q02);
                \path[->]   (-4,-1.5)   edge[bend right]    node[below,scale=1.5]   {$\cong$}   (2,-1.5);
                \path[->]   (-4.8,2)    edge[bend left ]    node[above,scale=1.5]   {$\cong$}   (4.8,-0.7);      
            \end{tikzpicture}
        \end{center}

    Since the graph is non-redundant, we thus have $t(e)=t(\pi(e))$ for each $e\in o^{-1}(v)\setminus\{e_0\}$. So it now suffices to show that $l(t(e_0))=l_p(w)$ and $l_p(v)=l(t(d_0))$.


    We will first show $l(t(e_0))\subseteq l_p(w)$ and $l(t(d_0))\subseteq l_p(v)$, for this. Take some path $p$ that ends in $t(e_0)$, again we may assume that $p=\tilde{p}e_0$ for some path $\tilde{p}$ ending in $v$, by \lemref{ch2-lem3}. Since $l(v)=l(w)$ we have some path $q$ that ends in $w$ s.t. $\mathcal{T}(G,R)^{\tilde{p}}\cong \mathcal{T}(G,R)^{q}$. Since $w\neq R$, $q$ is not empty, and thus we have a prefix $\tilde{q}$ s.t. $q=\tilde{q}d$ for an edge $d$ terminating in $w$. Since we have the bijection $\pi$, we can reassemble the isomorphism such that it takes $\tilde{p}e$ to $q\pi(e)$ for edges $e\neq e_0$ and thus has to take $p=\tilde{p}e_0$ to $\tilde{q}$. This gives us $\mathcal{T}(G,R)^{p}\cong \mathcal{T}(G,R)^{\tilde{q}}$. Therefore $l(t(e_0))\subseteq l_p(w)$, and by swapping $v$ and $w$ we can also show that $l(t(d_0))\subseteq l_p(v)$.

    For the converse inclusions $l_p(w) \subseteq l(t(e_0))$ and $l_p(v) \subseteq l(t(d_0))$ we take some path $\tilde{p}$ that ends in a predecessor of $v$. Then for an edge $e$ connecting it to $v$ define $p=\tilde{p}e$ a path ending in $v$. As $l(v)=l(w)$ we have some path $q$ ending in $w$ with $\mathcal{T}(G,R)^{p}\cong\mathcal{T}(G,R)^{q}$. As before, this isomorphism can be assumed to take $\tilde{p}$ to $qd_0$, thus gives us the isomorphism between $\mathcal{T}(G,R)^{\tilde{p}}$ and $\mathcal{T}(G,R)^{qd_0}$. This shows that $l_p(v) \subseteq l(t(d_0))$ and by swapping $v$ and $w$ in the above argument we also have $l_p(w) \subseteq l(t(e_0))$.

    This shows that $M_v=M_w$.
\end{proof}
This allows us to define $M_{l(v)}:=M_v$ without contradiction.
From the above proof we can see that 
\begin{cor}\label{ch2-cor1}
    For any non-redundant rooted graph $(G,R)$, with $G\setminus\{R\}$ strongly connected and any $v\in VG$, s.t. there exists $w\neq v$ s.t. $l(v)=l(w)$ we have $|l(v)|=1$
\end{cor}
\begin{proof}
    As in the proof of \lemref{ch2-lem5} we have for each path $p\in\mathcal{T}(G,R)$ terminating in $v$, we have some $e_p\in o^{-1}(v)$, $d_p\in o^{-1}(w)$ s.t. there is a bijection
        \[\pi_p:o^{-1}(v)\setminus\{e_p\}\to o^{-1}(w)\setminus\{d_p\}\]
    with $t(\pi_p(e))=t(e)$ for each $e\in o^{-1}(v)\setminus\{e_p\}$, and 
        \[\mathcal{T}(G,R)^{p}_{p_{-}}\cong \mathcal{T}(G,R)_{t(d_p)},\]
    for $p_{-}$ being the immediate prefix of $p$. Note that $t(e_p)\neq t(d_p)$ since then we could extend $\pi$ to all of $o^{-1}(v)$ inducing a n.e.c. relation. Now to show that $\mathcal{T}(G,R)^{p}\cong \mathcal{T}(G,R)^{q}$, for two paths with $T(p)=T(q)=v$ it suffices to show that $t(d_{p})=t(d_q)$ since then 
        \[\mathcal{T}(G,R)^{p}_{p_{-}}\cong\mathcal{T}(G,R)_{t(d_p)}\cong\mathcal{T}(G,R)_{t(d_q)} \mathcal{T}(G,R)^{q}_{q_{-}}.\]
    To see that assume that $t(d_p)\neq t(d_q)$. Let $k_q\in\mathds{Z}_{\geq 0}$ be the amounts of edges $e\in o^{-1}(v)$ with $t(e)=t(d_q)$, and analogously let $l_q\in\mathds{Z}_{\geq 0}$ be the amount of $e\in o^{-1}(w)$ with $t(e)=t(d_q)$. Now, since $t(e_q)\neq t(d_q)$ we have $l_q=k_q+1$, by looking at $\pi_q$. Additionally, if we look at $\pi_p$ since $t(d_p)\neq t(d_q)$ we have
        \[l_q=k_q+\delta_{t(e_p)=t(d_q)}.\]
    combining these equalities shows that $t(e_p)=t(d_q)$. Analogously we can show that $t(e_q)=t(d_p)$. However this allows us to calculate with multisets
	\[t_m(o^{-1}(v))\uplus t_m(o^{-1}(v))=t_m(o^{-1}(v)\setminus\{e_q\})\uplus\{t(e_q), t(e_p)\} \uplus t_m(o^{-1}(v)\setminus\{e_p\})=t_m(o^{-1}(w)\setminus\{d_q\})\uplus\{t(d_p),t(d_q)\} \uplus t_m(o^{-1}(w)\setminus\{d_p\})=t_m(o^{-1}(w))\uplus t_m(o^{-1}(w))\]
    ,where we interpret $\{t(e_q), t(e_p)\}$ and  $\{t(d_p),t(d_q)\}$ as multisets with all multiplicities equal to $1$. This shows that $t_m(o^{-1}(v))=t_m(o^{-1}(w))$  and thus there is an n.e.c. relation $\sim$ with $v\sim w$. This is however impossible since we assumed $v\neq w$ and $G$ non-redundant.
 
    So we must have $t(d_p)=t(d_q)$ and thus $\mathcal{T}(G,R)^{p}\cong \mathcal{T}(G,R)^{q}$.
\end{proof}

Note also that if the predecessors of a vertex have label with only one element, the vertex itself also has to have a label with only one vertex.

\begin{lemma}\label{ch2-lem5.5}
    For any non-redundant rooted graph $(G,R)$, s.t. $G\setminus \{R\}$ is strongly connected, then for any edge $e\in EG$ with $o(e)\neq R$ we have: 
        \[|l(o(e))|=1\implies |l(t(e))|= 1\]
\end{lemma}
\begin{proof}
    Take some paths $p_{+}$ and $q_{+}$ with $T(p_{+})=T(q_{+})=t(e)$. We will show that $\mathcal{T}(G,R)^{p_{+}}\cong \mathcal{T}(G,R)^{q_{+}}$, as in the proof of \lemref{ch2-lem3} we may assume that $p_{+}=pe$ and $q_{+}=qe$ for some $p,q\in\mathcal{T}(G,R)$. Since $T(p)=T(q)=o(e)$ we have an isomorphism:
        \[\phi:\mathcal{T}(G,R)^{p}\to\mathcal{T}(G,R)^{q}\]
    by \lemref{ch2-lem1}, we may assume that $\phi(pr)=qr$ for any $r\in\mathcal{T}(G,o(e))$ and thus in particular $\phi(p_{+})=\phi(pe)=qe=q_{+}$. So $\phi$ induces an isomorphism:
        \[\mathcal{T}(G,R)^{p_+}\cong \mathcal{T}(G,R)^{q_+}\]
\end{proof}

Now if we combine \lemref{ch2-lem3}, \coref{ch2-cor1} and \lemref{ch2-lem5.5} we get that if a vertex has two predecessors in $G\setminus\{R\}$, then it must have a label with cardinality one. And by recursively applying \lemref{ch2-lem5.5}, we see that the only way there can be labels with more than one element (in $G$ with $G\setminus \{R\}$ strongly connected) is that every vertex has a unique predecessor in $G\setminus\{R\}$. In this case, since $G\setminus\{R\}$ is strongly connected, the vertices of $G\setminus\{R\}$ can be assumed to be $v_1,v_2,v_3,\dots,v_N$ with edges only going from $v_i$ to $v_{i+1\bmod N}$. We will call such a graph \textbf{circular}.  This case will be considered separately.

In the non-circular case,each label has size $1$ and so we have by \coref{ch2-cor0}, that for any $v$, s.t. there are edges $e_1\in o^{-1}(R)$ $e_2\in o^{-1}(v)$ with $t(e_1)=t(e_2)$ we have $l(v)\cap l(R)\neq \emptyset$ and thus $l(v)=l(R)$. Showing the third property of \defref{ch2-def1}.

\begin{cor}\label{ch2-cor2}
    For any $v$ s.t. $t(o^{-1}(v))\cap t(o^{-1}(R))\neq \emptyset$ we have
        \[l_m(t_m(o^{-1}(R)))=l_m(t_m(o^{-1}(v)))\cup\{l_p(v)\}.\]
\end{cor}
\begin{proof}
    By \coref{ch2-cor0}, we have a path $p\in\mathcal{T}(G, R)$ that ends in $v$ s.t.
        \[\mathcal{T}(G, R)^p\cong \mathcal{T}(G, R)^{\varepsilon_R}=\mathcal{T}(G, R)\]
    Now if we restrict this isomorphism to the neighbourhood of $p$ (in $\mathcal{T}(G, R)^p$) we get a bijection
        \[\pi:\{pe\mid e\in o^{-1}(v)\}\cup\{p_-\}\to \{e\mid e\in o^{-1}(R)\},\]
    where $p_{-}$ is the direct predecessor of $p$, s.t. for each $q$ in the domain of $\pi$ we have
        \[\mathcal{T}(G,R)^q\cong \mathcal{T}(G,R)^{\pi(q)},\]
    which means that $l(T(q))\cap l(T(\pi(q)))\neq \emptyset$. So by the previous lemma, we have
        \[l(t_m(o^{-1}(R)))=l(t_m(o^{-1}(v)))\cup\{l(p_{-})\}\]
    and by noting that $p_{-}$ ends with a predecessor of $v$ it gives us our result.
\end{proof}

Finally, by combining these results, we get that $l$ is an actual labelling, i.e.
\begin{theorem}\label{ch2-thm1}
    For any non-circular rooted graph $(G,R)$, without edges s.t.
    \begin{itemize}
        \item $t^{-1}(R)=\emptyset$
        \item $G\setminus \{R\}$ is strongly connected and non-redundant
    \end{itemize}
    then if $\mathcal{T}(G,R)$ is cocompact, there is an actual labelling on $G$.
\end{theorem}
\begin{proof}
    For each $v\in VG\setminus\{R\}$, we will define $l(v)$ to be the set of all isomorphism types of $\mathcal{T}(G, R)^p$ for $p$ that ends in $v$. For each $v\in VG\setminus\{R\}$ we define $l_p(v)$ to be the label $l(o(e))$ for any edge with $o(e)=v$, this is well-defined by \lemref{ch2-lem3}. Using \lemref{ch2-lem5}, we can define $M_{l(v)}:=M_v$. The first $3$ points of the definition of an actual label follow from these definitions, whereas the last point follows from \coref{ch2-cor2}.
\end{proof}

We will be able to remove most of the restrictions on the graph from the theorem. In the circular graph case, we can construct the actual labelling directly. 
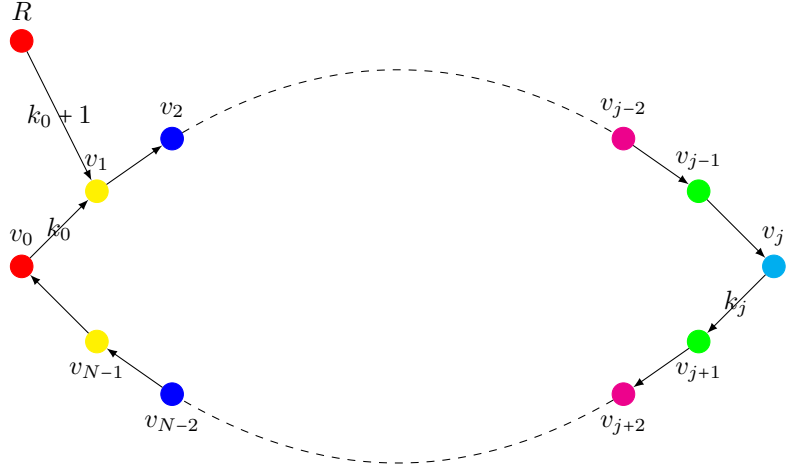
\begin{figure}[h]
    \centering
        \begin{tikzpicture}
            \node[big red node,label=above:{$R$}]                   (R) at (-5,3)               {};
            \node[big red node, label=above:{$v_0$}]                (v0) at (-5,0)              {};
            \node[big yellow node, label=above:{$v_1$}]             (v1) at (-4,1)              {};
            \node[big blue node, label=above:{$v_2$}]               (v2) at (-3,1.7)            {};
            \node[big magenta node, label=above:{$v_{j-2}$}]        (vj-2) at (3,1.7)           {};
            \node[big green node, label=above:{$v_{j-1}$}]          (vj-1) at (4,1)             {};
            \node[big cyan node, label=above:{$v_j$}]               (vj) at (5,0)               {};
            \node[big green node, label=below:{$v_{j+1}$}]          (vj+1) at (4, -1)           {};
            \node[big magenta node, label=below:{$v_{j+2}$}]        (vj+2) at (3,-1.7)          {};
            \node[big blue node, label=below:{$v_{N-2}$}]           (vN-2) at (-3,-1.7)         {};
            \node[big yellow node, label=below:{$v_{N-1}$}]         (vN-1) at (-4,-1)           {};
        
            \path[-latex]
                    (R)             edge[]              node[]                  {$k_0+1$}       (v1)
                    (v0)            edge[]              node[]                  {$k_0$}         (v1)
                    (v1)            edge[]              node[]                  {}              (v2)
                    (vj-2)          edge[]              node[]                  {}              (vj-1)
                    (vj-1)          edge[]              node[]                  {}              (vj)
                    (vj)            edge[]              node[]                  {$k_j$}         (vj+1)
                    (vj+1)          edge[]              node[]                  {}              (vj+2)
                    (vN-2)          edge[]              node[]                  {}              (vN-1)
                    (vN-1)          edge[]              node[]                  {}              (v0)
                    ;
            \draw[dashed] (v2) to[bend left] (vj-2);
            \draw[dashed] (vj+2) to[bend left] (vN-2);

        \end{tikzpicture}
    \caption{All non-redundant circular graphs with cocompact unfolding tree are of this form, where the edges have multiplicity of their label and the unlabelled edges have multiplicity $1$, the colour of the vertices represents their label and $j=\frac{N}{2}$}
\end{figure}
\begin{lemma}\label{ch2-lem5.6}
    For any circular rooted graph $(G,R)$, with cocompact unfolding tree s.t. $G\setminus\{R\}$ is non-redundant and $|o^{-1}(R)|> 2$, we have an actual labelling on $G$.
\end{lemma}
\begin{proof}
    Write $VG\setminus\{R\}=\{v_0,v_1,\dots, v_{N-1}\}$, s.t. each edge that originates in $v_i$ ends in $v_{i+1\bmod N}$, for every $0\leq i \leq N-1$, additionally denote $k_i:=|o^{-1}(v_i)|$.

We first deal with the case where the graph contains only two vertices, i.e. $N=0$. If we assume that $|o^{-1}(R)|\neq k_0 +1$, then for any path $p\neq\varepsilon_R$ the only non-root vertex in $\mathcal{T}(G,R)^p$ that doesen't have degree $k_0$ is $\varepsilon_R$. So if $|p_1|\neq |p_2|$ we cannot have $\mathcal{T}(G,R)^{p_1}\cong \mathcal{T}(G,R)^{p_1}$ since in these trees $\varepsilon_R$ has different distances from the root. Since in $\mathcal{T}(G,R)$ there are paths of any length, $\mathcal{T}(G,R)$ cannot be cocompact. So we can assume that $|o^{-1}(R)|= k_0 +1$, in this setting we can see that setting $l(R)=l(v_0)=0$ induces an actual labeling.

 So now we may assume that $N>0$ and $v_1\in t(o^{-1}(R))$ and thus by \coref{ch2-cor0} there is a path $p\in\mathcal{T}(G,R)$  that ends in $v_0$ s.t. 
        \[\mathcal{T}(G,R)\cong\mathcal{T}(G,R)^p,\]
    which, after restricting to the outgoing neighbourhood of the root and taking out the edge $e_0$ that is mapped to the immediate prefix of $p$, gives us a bijection $\pi:o^{-1}(v_0)\to o^{-1}(R)\setminus\{e_0\}$ s.t. $t(\pi(e))=t(e)=v_1$. With some $e_0\in o^{-1}(R)$ s.t.
        \[\mathcal{T}(G,R)^{p}_{p_{-}}\cong \mathcal{T}(G,R)_{t(e_0)}.\]
    This means that all but at most one of the edges that originate in $R$ end in $v_1$. Now take $i$ s.t. $t(e_0)=v_i$, then, by the same argument, all but one of the edges that originate in $R$ terminate in $v_i$, but since $|o^{-1}(R)|>2$ we must have $v_i=v_1$. Thus, all edges originating in $R$ end in $v_1$.

    Additionally, since $t(e_0)=v_1$ we have 
        \[\mathcal{T}(G,R)^{p}_{p_{-}}\cong \mathcal{T}(G,R)_{v_1}.\]
    Let $\phi$ be an isomorphism between these trees. Note also that since $O(p)=R$ and $T(p)=v_0$, $p$ must pass through every vertex of $G$. Additionally, since $G\setminus\{R\}$ is non-redundant and has more than one vertex, we cannot have $k_i=1$ for all $i\leq N$, so take $j$ to be the largest index s.t. $k_j>1$. Now let $q$ be the longest prefix of $p$ s.t. $T(q)=v_j$, we can write $p_{-}=qe_{j+1}e_{j+2}\dots e_{N-1}$, with $t(e_i)=v_i$, for $j<i<N$, and note that $qe_{j+1}e_{j+2}\dots e_{N-l}$ is the unique element in $\mathcal{T}(G,R)^{p}_{p_{-}}$ with distance $l$ from the root, for $l\leq N-j$. Take edges $d_l$ for $2\leq l \leq N-j$ with $t(d_l)=v_l$ s.t. $\phi(qe_{j+1}e_{j+2}\dots e_{N-l})=d_2\dots d_{l}$ and $\phi(p_{-})=\varepsilon_{v_1}$. Thus, we also have $|o^{-1}(v_i)|=1$ for each $1\leq i < N-j$ and $|o^{-1}(v_{N-j})|>1$. Now take some edge $\tilde{e}_{j+1}\in o^{-1}(v_j)\setminus\{e_{j+1}\}$, then in turn have some edges $d_{N-j+1}\in o^{-1}(v_{N-j})$ s.t. $\phi(q\tilde{e}_{j+1})=d_2\dots d_{N-j}d_{N-j+1}$ (as $q\tilde{e}_{j+1}$ has distance  $N-j+1$ from $p_{-}$ in $\mathcal{T}(G,R)^{p}_{p_{-}}$) and thus
        \[\mathcal{T}(G,R)_{v_{j+1}}\cong\mathcal{T}(G,R)^{p}_{q\tilde{e}_{j+1}}\cong \mathcal{T}(G,R)_{d_1d_2\dots d_{N-j}d_{N-j+1}}\cong \mathcal{T}(G,R)_{v_{N-j+1}}.\]
    So since $G\setminus\{R\}$ is non-redundant we have to have $v_{j+1}=v_{N-j+1}$ so $j=N-j$(i.e. $j=\frac{N}{2}$, making $N$ even), so for each $1\leq i < j$, $|o^{-1}(v_i)|=1$ and by assumption for each $j<i\leq N$, $|o^{-1}(v_i)|=1$. Thus, $v_0$ and $v_j$ are the only vertices in the graph that have more than one outgoing edge, with the distance from $v_0$ to $v_j$ and $v_j$ to $v_0$ being the same (as $j=N-j$).
    
    We will construct the actual labelling on the graph now by setting  $l(v_i):=i$ and $l(v_{N-i}):=i$ for each $0\leq i\leq j$, and $l(R):=l(v_0)=0$. Since $j=N-j$, this defines the labelling on all vertices of $G$.  Note that if $l(v_l)=l(v_k)$, for $l<k$ we must have $k=N-l$. Of course, we define $l_p(v_i)=l(v_{i-1\bmod N})$, and since each vertex (except for $v_1$) only has one predecessor the third point of \defref{ch2-def1} follows trivially (for $v_1$ it follows since $l(R)=l(v_0)$). Additionally, we set for any $1\leq i\leq j-1$ $M_i:=\{(i+1,1),(i-1,1)\}$, $M_0:=\{(1,k_0+1)\}$ and $M_j=\{(j-1,k_j+1)\}$ (where the second component indicates the multiplicity of the element). Now, since for each $1\leq i\leq j-1$ $l_p(v_i)=i-1$ and has the unique outgoing neighbour $v_{i+1}$ with only one edge pointing towards it, we have $l(t_m(o^{-1}(v_i)))=M_{l(v_i)}\setminus\{l_p(v_i)\}$. Furthermore, for $0\leq i\leq j$ we have $l_p(v_{N-i})=l(v_{N-(i+1)})=i+1$ and the only neighbour of $v_{N-i}$ is $v_{N-(i-1)}$ (with one vertex pointing into it), has label $i$, so
        \[l(t_m(o^{-1}(v_{N-i})))=M_{i+1}\setminus\{i\}=M_{l(v_{N-i})}\setminus\{l_p(v_{N-i})\}.\]
    Since $l_p(v_0)=l(v_N)=l(v_1)=1$ and its unique neighbour has $k_0$ edges pointing to it, we also get $l(t_m(o^{-1}(v_0)))=M_{l(v_0)}\setminus\{l_p(v_0)\}$. Analogously, since $l(v_{j+1})=l(v_{j-1})=j-1=l_p(v_j)$ we also have $l(t_m(o^{-1}(v_j)))=M_{l(v_j)}\setminus\{l_p(v_j)\}$. The second point from \defref{ch2-def1} thus follows.
    Finally, since all edges originating in $R$ point towards $v_1$ and $|o^{-1}(R)|=k_0+1$ since
        \[\mathcal{T}(G,R)\cong\mathcal{T}(G,R)^p,\]
    we must have $l(t_m(o^{-1}(R)))=\{(1,k_0+1)\}=M_{0}=M_{l(R)}$.
    Thus, we have constructed the required labelling. 
\end{proof}
Note that from the proof of this lemma, we can see that we for circular graphs $(G, R)$ that have a cocompact unfolding tree has $N=|VG\setminus\{R\}|$ even and only $3$ vertices can have outgoing degree $>1$: $R,v_1,v_{N/2}$. Furthermore, $v_1$ is the only outgoing neighbour of $R$.

This result removes the non-circularity condition from \thmref{ch2-thm1}, making it evolve into the following.

\begin{customthm}{1.1}
    For any rooted graph $(G,R)$, without edges s.t.
    \begin{itemize}
        \item $t^{-1}(R)=\emptyset$
        \item $|o^{-1}(R)|>2$
        \item $G\setminus \{R\}$ is strongly connected and non-redundant
    \end{itemize}
    then if $\mathcal{T}(G,R)$ is cocompact, there is an actual labelling on $G$.
\end{customthm}
Note the condition $|o^{-1}(R)|>2$ is only relevant for circular graphs.

We can also remove the non-redundancy condition by noting that by \lemref{ch2-lem0.6} if a quotient of a graph has an actual label, so does the graph itself. Thus, if a graph has a cocompact unfolding tree and reduces to a strongly connected graph, it has an actual labelling.


 To remove the strong connectivity condition, we will have to consider focal unfolding cycles separately. They are defined as follows.
 
 \begin{defn}
    For any $N\in \mathds{Z}_{\geq 1}$ and any $(k_i)^{N-1}_{i=0}\in (\mathds{Z}_{\geq 1})^N$ we define the focal unfolding cycle $\mathbf{FC}_{(k_i)^{N-1}_{i=0}}$ with:
    \begin{align*}
        V\mathbf{FC}_{(k_i)^{N-1}_{i=0}}:=  &\{v_0,\dots,v_{N-1}\}\cup\{w_0,\dots, w_{N-1}\}\cup \{R\},\\
        E\mathbf{FC}_{(k_i)^{N-1}_{i=0}}:=  &\{(v_i,v_{i+1\bmod N}, j)\mid 0\leq i< N,\ 1\leq j\leq k_i\}\cup\\
                                            &\{(w_i,v_{i+1\bmod N},j)\mid 0\leq i< N,\ 1\leq j\leq k_i-1 \}\cup\\
                                            &\{(w_i,w_{i-1\bmod N},0)\mid 0\leq i< N\}\cup\\
                                            &\{(R,v_1,i)\mid 1\leq i\leq k_1\}\cup\{(R,w_{N-1},1)\}\text{ and }\\
        \forall v,w\in V\mathbf{FC}&_{(k_i)^{N-1}_{i=0}},\forall j\in \mathds{Z}_{\geq 1}\ o((v,w,j)):=v,\ t((v,w,j)):=w,
    \end{align*}    
    where the last definition only applies if the edge is defined. We will also assume that $k_0>1$, to make sure that the degree of the root is greater than $2$.
 \end{defn}
 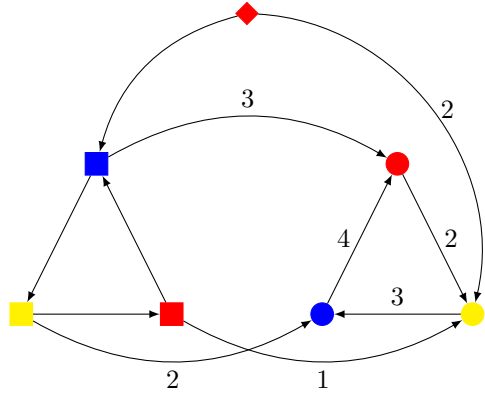
\begin{figure}[h]
     \centering
     \begin{tikzpicture}
            \node[diamond red node]             (R) at (0,3)            {};
            \node[big red node]                 (v1) at (2,1)               {};
            \node[big yellow node]              (v2) at (3,-1)              {};
            \node[big blue node]                (v3) at (1,-1)             {};
            \node[square blue node]             (v3*) at (-2,1)               {};
            \node[square red node]              (v1*) at (-1,-1)              {};
            \node[square yellow node]           (v2*) at (-3,-1)             {};
        
            \path[-latex]
                    (v1)            edge[]              node[right]              {$2$}            (v2)
                    (v2)            edge[]              node[above]              {$3$}            (v3)
                    (v3)            edge[]              node[left]               {$4$}            (v1)
                    (v1*)           edge[]              node[above]              {}               (v3*)
                    (v3*)           edge[]              node[above]              {}               (v2*)
                    (v2*)           edge[]              node[above]              {}               (v1*)
                    (v1*)           edge[bend right]    node[below]              {$1$}            (v2)
                    (v2*)           edge[bend right]    node[below]              {$2$}            (v3)
                    (v3*)           edge[bend left]     node[above]              {$3$}            (v1)
                    (R)             edge[bend left=50]  node[right]              {$2$}            (v2)
                    (R)             edge[bend right]    node[above]              {}               (v3*)
                    ;
     \end{tikzpicture}
     \caption{A focal unfolding cycle associated with the sequence $(2,3,4)$}
     \label{fig:enter-label}
 \end{figure}
The Focal unfolding cycles constitute an exception since they can be non-redundant and not strongly connected, but admit an actual labelling  and therefore have a cocompact unfolding tree. This can, however, be shown to be the only such exception, giving us the following result. 

\begin{lemma}\label{ch2-lem6}
    For each rooted graph $(G,R)$ without sinks, s.t. 
    \begin{itemize}
        \item $t^{-1}(R)=\emptyset$
        \item $G\setminus \{R\}$ is non-redundant.
        \item $\mathcal{T}(G,R)$ is cocompact and not isomorphic to the unfolding tree of a focal unfolding cycle.
    \end{itemize}
    $G\setminus \{R\}$ is strongly connected.
\end{lemma}
The following version of the proof was pointed out to me by my supervisor, Colin Reid.
\begin{proof}
    We note that the group of undirected automorphisms (i.e. the group of automorphisms ignoring the direction of the edges) of a cocompact tree has finitely many orbits (since two vertices have isomorphic trees rooted on them if and only if there is an undirected isomorphism mapping them to one another). Therefore, the local action diagram (as defined in \cite{Reid2020}*{Definition 1.1.}) of this group is finite. Since the tree this group is acting on does not have leaves, the local action diagram does not have sinks. Thus, the local action diagram  has no horocyclic ends, stray half trees or stray leaves. Additionally, since the tree that a focal cycle action diagram acts on is the unfolding tree of a focal unfolding cycle, the local action diagram cannot be a focal cycle. So by \cite{Reid2020}*{Corollary 5.9., Proposition 5.10.} we can see that the undirected automorphism group acts geometrically dense.  This allows us to apply \cite{Reid2020}*{Lemma 2.3.} to show that for any directed edges $d,e\in E\mathcal{T}(G,R)$ there is an undirected automorphism $g$ s.t. $\mathcal{T}(G,R)_{t(ge)}\subseteq\mathcal{T}(G,R)_{t(d)}$. This means we have for any $p,q\in \mathcal{T}(G,R)$ there is some $p'\in\mathcal{T}(G,R)_q$ s.t. $\mathcal{T}(G,R)_{p'}\cong\mathcal{T}(G,R)_p$. Now if we take some vertices $v,w\in VG$ and $p,q\in \mathcal{T}(G,R)$ with $T(p)=v, T(q)=w$, we have as before some $p'\in\mathcal{T}(G,R)_q$ s.t. $\mathcal{T}(G,R)_{p'}\cong\mathcal{T}(G,R)_p$. So since $G$ is non-redundant we have $T(p')=T(p)=v$ and can write $p'=qr$, for some path $r$, since it is in the half tree rooted at $q$. Since $T(q)=v$, $r$ is the desired path from $w$ to $v$. This gives us strong connectivity of $G$.

\end{proof}
To deal with the case of focal unfolding cycles, we will first show that they in fact have an actual labelling and then that any factor of a focal unfolding cycle over a non-edge-collapsing equivalence relation is a focal unfolding cycle or circular. Thus, any non-redundant graph, which has the same unfolding tree as an unfolding focal cycle, is a focal unfolding cycle or a circular graph (with cocompact unfolding tree) and thus has an actual labelling.

\begin{lemma}\label{ch2-lem7}
    Any focal unfolding cycle has an actual labelling.
\end{lemma}
\begin{proof}
    Take $G=\mathbf{FC}_{(k_i)^{N-1}_{i=0}}$, for some $N\in\mathds{Z}_{\geq 1}$, $(k_i)^{N-1}_{i=0}\in(\mathds{Z}_{\geq 1})^N$. We will label $l(v_i):=i$ and $l(w_i):=i$ for each $0\leq i<N$ and $l(R)=0$. The predecessors of each $v_i$ are either $\{v_{i-1\bmod N}, w_{i-1\bmod N}\}$, if $i\neq 1$ or $\{v_0,w_0,R\}$ if $i=1$, and  the predecessors of each $w_i$ are either $\{w_{i+1\bmod N}\}$, if $i\neq N-1$ or $\{w_0,R\}$, if $i=N-1$. So, setting the previous label as $l_p(v_i):=i-1\bmod N$ and $l_p(w_i)=i+1\bmod N$, for each $i$, gives us the third property from \defref{ch2-def1}. The outgoing neighbourhood of each $v_i$, consists of just $v_{i+1\bmod N}$ with multiplicity $k_i$, so we will set $M_i:=\{(i+1\bmod N, k_i),(i-1\bmod N, 1)\}$. Since the outgoing neighbourhood of each $w_i$ consists of $w_{i-1\bmod N}$ with multiplicity $1$ and $v_{i+1\bmod N}$ with multiplicity $k_i-1$, the label multisets satisfy the second condition from \defref{ch2-def1}. The outgoing neighbourhood of the root consists of $w_{N-1}$ with multiplicity $1$ and $v_1$ with multiplicity $k_i$, thus the fourth property of \defref{ch2-def1}, also follows. The first property of the definition follows by inspection of the $M_i$'s.
\end{proof}

To see that the non-redundant graphs that have an unfolding tree isomorphic to a unfolding tree of a focal unfolding cycle also have an actual labelling, we will simply show that it must also be a focal unfolding cycle, or a circular graph. To that end, we define periodic sequences of numbers as follows.
\begin{defn}
    For any $N\geq 1$, a sequence of numbers $(k_i)^N_{i=0}\in (\mathds{Z}_{\geq 1})^N$ is\\ ($m$-)periodic if there is some $1\leq m<N$ s.t.
        \[\forall 1\leq i\leq N,\ k_{i}=k_{i+m\bmod N}.\]
\end{defn}
Note that, due to Bézout's identity\cite{bachet1874problemes}, we can assume that  $m$ is a divisor of $N$.
We will also consider sequences that have a lot of their components $=1$, which we define as follows.
\begin{defn}
    For any $N\geq 0$, a sequence $(k_i)^{N-1}_{i=0}\in (\mathds{Z}_{\geq 1})^N$ is \textbf{full of ones}, if:
    \begin{itemize}
        \item for $N$ even, $k_i=1$, whenever $i\neq 0,\frac{N}{2}$, and
        \item for $N$ odd , $k_i=1$, whenever $i\neq 0$.
        
    \end{itemize}
\end{defn}
If the above two types of sequences are excluded, we can show that the focal unfolding cycle graphs based on these sequences are non-redundant.  
\begin{lemma}\label{ch2-lem8}
    For any $N\geq 1$ and any non-periodic sequence $(k_i)^N_{i=0}$ in $\mathds{Z}_{\geq 1}$, the graph $\mathbf{FC}_{(k_i)^{N-1}_{i=0}}$ is:
    \begin{itemize}
        \item reducible to a circular graph, if  $(k_i)^N_{i=0}$ is full of ones,
        \item non-redundant, otherwise.
    \end{itemize}
\end{lemma}
\begin{proof}
    For convenience, we will write $G=\mathbf{FC}_{(k_i)^{N-1}_{i=0}}$.\\
   First, we will consider $(k_i)^{N-1}_{i=0}$ being full of ones. We will define an equivalence relation on $G$, by setting $w_i\sim v_{N-i\bmod N}$ for $0\leq i<N$ and pairing up their outgoing edges arbitrarily. If $k_i=k_{N-i}=1$ then $w_i$ has unique outgoing neighbour $w_{i-1}$ with one edge going towards it and $v_{N-i\bmod N}$ has unique outgoing neighbour $v_{N-(i-1)\bmod N}$ with one edge going towards it, so at these vertices the non-edge collapsing condition holds. If that is not the case, then,since $(k_i)^{N-1}_{i=0}$ is full of ones, $i=0$ or $i=N/2$. If $i=0$, $w_0$ has outgoing degree $k_0$ with neighbours $w_{N-1}$ and $v_1$, on the other hand, $v_0$ has outgoing degree also $k_0$ with unique neighbour $v_1$, since $w_{N-1}\sim v_1$ the non-edge collapsing condition holds. If $i=N/2$ (assuming N is even), similarly the outgoing degree of $w_i$ is $k_i$ with neighbours $w_{i-1}$ and $v_{i+1}$, on the other hand, $v_i$ has outgoing degree also $k_i$ with unique neighbour $v_{i+1}$, note that since $2i=N$, $i-1=2i-(i+1)=N-(i+1)$, we have $w_{i-1}\sim v_{i+1}$, showing that the non-edge-collapsing condition still holds. Thus $\sim$ is a non-edge-collapsing equivalence relation. Note that the vertices of $G/{\sim}$ are $\{[R]_{\sim},[v_0]_{\sim},\dots,[v_{N-1}]_{\sim}\}$, with edges going exclusively from $[v_i]_{\sim}$ to $[v_{i+1\bmod N}]_{\sim}$ and from $[R]_{\sim}$ to $[v_{1\bmod N}]_{\sim}$. Therefore, $G/{\sim}$ is circular.

    Now, if $(k_i)^{N-1}_{i=0}$ is not full of ones we will take some non-edge-collapsing equivalence relation $\sim$ on $G\setminus\{R\}$ and show that it cannot be non-trivial. Since $G\setminus\{R\}$ does not change under cyclic permutations of the $k_i$'s we may assume that $k_0\neq 1$. We consider several cases when the relation is non-trivial and show that they all lead to a contradiction.

    \begin{enumerate}
        \item  If $\exists j<k,\ v_j\sim v_k$ we have as the only outgoing neighbour of $v_j$ and $v_k$ is $v_{j+1\bmod N}$ and $v_{k+1\bmod N}$ respectively, we must have $v_{j+1\bmod N}\sim v_{k+1\bmod N}$. So inductively we have for each $l\geq 0$, $v_{j+l\bmod N}\sim v_{k+l\bmod N}$. Thus, we must have for $m:=k-j$ and for any $1\leq i\leq N$, $v_{i}\sim v_{i+m\bmod N}$. Since $\sim$ preserves the outgoing degree, we must have $k_{i}= k_{i+m\bmod N}$ and thus the sequence $(k_i)^N_{i=0}$ is periodic, contradicting our assumption.
        \item If $\exists 0\leq j,k\leq N-1,\ w_j\sim v_k$, we note that since $w_{j-1\bmod N}$ is an outgoing neighbour of $w_j$, and $v_{k+1\bmod N}$ is the only outgoing neighbour of $v_k$, we must have $w_{j-1\bmod N}\sim v_{k+1\bmod N}$. So inductively we have for each $m>0$, $w_{j-m\bmod N}\sim v_{k+m\bmod N}$. Thus, for each $w_{j'}$, there is some $v_{k'}$ s.t. $w_{j'}\sim v_{k'}$. Since we can pick $j$ to be any index, we will assume that $j=0$ and note that $k_0\neq 1$ (by assumption).
        We will now distinguish the cases where $k\neq 0$ and where $k=0$.
            \begin{enumerate}
                \item If $k\neq 0$, we must have $v_1\sim v_{k+1\bmod N}$, since $v_1$ is a neighbour of $w_0$ (since by assumption $k_0\geq 2$) and $v_{k+1\bmod N}$ is the unique neighbour of $v_k$. Since $1\neq k+1\bmod N$ we get a contradiction the same way as in the first point.
                \item If $k=0$, we will first assume that for any $l\neq m$, $v_l\not\sim v_m$, since this case was already covered in the first point. Since $w_{N-1}$ is a neighbour of $w_0$, and $v_1$ is the unique neighbour of $v_0$ we must have $v_1\sim w_{N-1}$. Repeating this argument inductively, we see that $w_i\sim v_{N-i}$ for any $0\leq i<N$. Now if $k_i\geq 2$ then $v_{i+1}$ is a neighbour of $w_i$ and thus we must have $v_{i+1}\sim v_{N-i+1}$, so by assumption we must have $i+1=N-i+1\mod N$. This gives us $2i=0\mod N$ and thus $i=0$ or $i=N/2$. So the sequence $(k_i)^{N-1}_{i=0}$ must be full of ones, contradicting our assumption. 
            \end{enumerate}
        \item If $\exists j<k<N,\ w_j\sim w_k$, we can exclude the second point, i.e. $\forall l,m,\ w_l\not\sim v_k$. Since $w_{j-1}$ is an outgoing neighbour of $w_j$ and the outgoing neighbours of $w_k$ are $w_{k-1}$ and (potentially) $v_{k+1}$, we must have $w_{j-1}\sim w_{k-1}$. By inductively repeating this argument, we have for $m=k-j$ and for any $0\leq i<N$, $w_{i}\sim w_{i+m\bmod N}$. Since $\sim$ preserves the outgoing degree, this also means: $k_{i}=k_{i+m\bmod N}$, making the sequence $(k_i)^{N-1}_{i=0}$ periodic and thus contradicting our assumption.
    \end{enumerate}
     So we see that any non-edge collapsing equivalence relation $\sim$ on $G$ must be trivial on $G\setminus\{R\}$. Additionally if we have $R\sim u$ for some vertex $u$, $R$ and $u$ must share their outgoing neighbourhood (as the relation is trivial on it). However this is only the case if $u=R$. So $\sim$ must be trivial showing that $G$ is non-redundant. 
   
    \end{proof}

We will now show that we can reduce every focal cycle to a focal cycle of a non-periodic sequence.

\begin{lemma}\label{ch2-lem9}
    For any $N-1$ and any sequence $(k_i)^{N-1}_{i=0}$ there is $M\leq N$, s.t. $(k_i)^{M-1}_{i=0}$ is non-periodic and there exists a non-edge-collapsing equivalence $\sim$ s.t.
        \[\mathbf{FC}_{(k_i)^{N-1}_{i=0}}/{\sim}\cong \mathbf{FC}_{(k_i)^{M-1}_{i=0}}.\]
\end{lemma}
\begin{proof}
    Take $M>0$ to be the minimal number s.t. $k_{i}=k_{i+M\bmod N}$ for each $0\leq i<N$. Note that $M$ divides $N$ and, by minimality, the sequence $(k_i)^{M-1}_{i=0}$ is non-periodic. To show that there is a non-edge-collapsing equivalence relation, we will construct a surjective homomorphism
        \[\pi:\mathbf{FC}_{(k_i)^{N-1}_{i=0}}\to \mathbf{FC}_{(k_i)^{M-1}_{i=0}},\]
    that is bijective at the outgoing neighbourhood of each vertex. Then by \lemref{lem-1} we have the required equivalence relation. We will set for each $0\leq i <N$
        \[\pi(v_i):=v_{i\bmod M},\quad \pi(w_i):=w_{i\bmod M}\text{  and  } \pi(R):=R,\]
    and for any $v,w\in V\mathbf{FC}_{(k_i)^{N-1}_{i=0}}$, $l\in \mathds{Z}_{\geq 1}$
        \[\pi((v,w,l))=(\pi(v),\pi(w),l),\]
    whenever $(v,w,l)\in E\mathbf{FC}_{(k_i)^{N-1}_{i=0}}$. This is well-defined since $k_i=k_{i\bmod M}$ for each $i$. Clearly, $\pi$ is  a graph homomorphism. Since $M\leq N$, $\pi$ is surjective. For any vertex $v\in V\mathbf{FC}_{(k_i)^{N-1}_{i=0}}$, we can see that $\pi|_{o^{-1}(v)}$ is a a bijection into $o^{-1}(\pi(v))$, by checking the cases $v=v_i$, $v=w_i$ and $v=R$ for $i\leq N$. Thus $\pi$ is a bijection on each outgoing neighbourhood, so \lemref{lem-1} gives us the required non-edge collapsing equivalence relation.
\end{proof}

The two above results allow us to deal with graphs whose unfolding trees are isomorphic to the unfolding tree of a focal cycle.
\begin{lemma}\label{ch2-lem10}
    Any non-redundant rooted graph $(H,S)$, with a number sequence $(k_i)^{N-1}_{i=0}$ s.t.
        \[\mathcal{T}(H,S)\cong \mathcal{T}(\mathbf{FC}_{(k_i)^{N-1}_{i=0}},R)\]
    admits an actual labelling.
\end{lemma}
\begin{proof}
    By \lemref{ch2-lem9} we may assume that $(k_i)^{N-1}_{i=0}$ is non-periodic. As the unfolding trees are isomorphic and $H$ non-redundant, we must have a non-edge-collapsing equivalency $\sim$ s.t.
        \[H\cong \mathbf{FC}_{(k_i)^{N-1}_{i=0}}/{\sim}.\]
     By \lemref{ch2-lem8} we then see that $H$ is either a focal cycle, or a circular graph. In the former case \lemref{ch2-lem7}, gives us an actual labelling and in the latter \lemref{ch2-lem5.6} does (note that the unfolding tree is cocompact since the unfolding tree of a focal cycle is).
\end{proof}
Combining \lemref{ch2-lem6} and \lemref{ch2-lem10} we can get rid of the strong connectivity condition from \thmref{ch2-thm1}. \Lemref{ch2-lem0.6} also allows us to remove the non-redundancy condition. The theorem now looks like this.

\begin{customthm}{1.2}
    For any rooted graph $(G,R)$, without edges s.t.
    \begin{itemize}
        \item $t^{-1}(R)=\emptyset$
        \item $|o^{-1}(R)|>2$
    \end{itemize}
    then if $\mathcal{T}(G,R)$ is cocompact, there is an actual labelling on $G$.
\end{customthm}


To expand the result to rooted graphs with root of degree $2$, we can use a broader result

\begin{lemma}
    For any  rooted graph $(G,R)$, s.t.  $t^{-1}(R)=\emptyset$, 
    then for any $r\in\mathcal{T}(G,R)$, $G$ has an actual labelling if and only if $G^r$ has. 
\end{lemma}
\begin{proof}
    Firstly, if $G$ has an actual labelling $(l,l_p)$, we can expand it to $G^r$ by setting $l(q)=l(T(q))$ for any $q\preceq r$ and $l_p(q)$ to be the label of its unique predecessor in $G^r$ for $q\neq r$. The fact that this labelling is actual follows from the definition of $G^r$(\defref{ch2-def-1}).

    We will prove the converse implication by induction over the length of $r$. First, note that since $G^{\varepsilon_R}\cong G$ the proposition follows length $0$. Now if we write $r=r_{-}e$ for some $e\in EG$ and the longest proper prefix $r_{-}\precneq r$ we will show that if $G^{r}$ has an actual labelling then so does $G^{r_{-}}$. First note that in $G^{r}$, $r$ has an outgoing neighbour in $VG\setminus\{R\}$ since $G$ has no sinks. Now note that $T(r)\in VG$ shares this neighbour with $r$ and therefore $l(r)=l(T(r))$. Since $r$ is the root of $G^r$ we must have
        \[l(t_m(o^{-1}(r)))=M_{l(r)}=M_{l(T(r))}=l(t_m(o^{-1}(T(r))))\cup\{l_p(T(r))\}.\]
    And since:
    \[t_m(o^{-1}(r))=t_m(o^{-1}(T(r)))\cup\{r_{-}\}\]
    we must have $l(r_{-})=l_p(T(r))=l(T(r_{-}))$. Now since $VG^r=VG^{r_{-}}\cup \{r\}$ we can restrict the labelling of $G^{r}$ onto $G^{r_{-}}$. Note that the properties of an actual labelling, away from the root $r^{-}$, follow directly since they hold in $G^{r}$. Now to see that they hold at $r_{-}$ we note first that a vertex $v$ in $G^{r_{-}}$ that has $r_{-}$ as a predecessor either has $T(r_{-})$ as predecessor as well and thus $l_p(v)=l(T(r_{-}))=l(r_{-})$, or $v$ is the longest proper prefix of $r_{-}$, in which case $r_{-}$ is its only predecessor. Additionally, we note that in $G^{r_{-}}$ $r_{-}$ has all the same outgoing edges as in $G^{r}$ except for the addition of an edge connecting it to $T(r)$ and since $l(T(r))=l(r)=l_p(r_{-})$ and the labelling being actual on $G^{r}$ we have
        \[l(t_m(o^{-1}(r_{-})))=M_{l(r_{-})}\setminus\{l_p(r_{-})\}\cup l(T(r))=M_{l(r_{-})}.\]
    This shows that $l$ is also an actual labelling on $G^{r_{-}}$. 
    Now inductively we may conclude that if $G^{r}$ has an actual labelling, so does $G^{\varepsilon_{R}}\cong G$.

\end{proof}
    If we now look at a rooted graph $(G,R)$ without sinks, $t^{-1}(R)=\emptyset$ and $|o^{-1}(R)|=2$, then we either have for each $v\in VG\setminus\{R\}$ $|o^{-1}(v)|=1$  or we have a path $r$ s.t. $G^r$ has a root of outgoing degree $>2$ and therefore has an actual labelling. In the first case, setting every vertex to be the same label gives us an actual labelling on the graph. For the second case, we can use the above lemma to show that $G$ also has actual labelling. This replaces the condition of $|o^{-1}(R)|>2$ in \thmref{ch2-thm1} with $|o^{-1}(R)|\geq 2$.
    Additionally, if $|o^{-1}(R)|=1$ and $G$ has no sinks then $\mathcal{T}(G,R)$ cannot be cocompact since then $\varepsilon_R$ is the only vertex with incoming + outgoing degree = $1$ and so we cannot have $\mathcal{T}(G,R)^{p}\cong \mathcal{T}(G,R)^{q}$ if $q$ and $p$ have different distances from $\varepsilon_{R}$, i.e. different lengths. And since $G$ has no sinks there are paths of unlimited length, making cocompactness impossible. This allows us to remove the $|o^{-1}(R)|>2$ condition from \thmref{ch2-thm1}, evolving it into the following.

    \begin{customthm}{1.3}
    For any rooted graph $(G,R)$, without sinks s.t.
    \begin{itemize}
        \item $t^{-1}(R)=\emptyset$
    \end{itemize}
    then if $\mathcal{T}(G,R)$ is cocompact, there is an actual labelling on $G$.
\end{customthm}

    Lastly, to remove the condition of $t^{-1}(R)=\emptyset$ we can show that it is implied by $\mathcal{T}(G,R)$ being cocompact.
    \begin{lemma}
        Let $(G,R)$ be a rooted graph without sinks. If $\mathcal{T}(G,R)$ is cocompact then $t^{-1}(R)=\emptyset$.
    \end{lemma}
    \begin{proof}
        We will assume that $\mathcal{T}(G,R)$ is cocompact and $t^{-1}(R)\neq\emptyset$.\\
        Define $\bar{G}$ to be an expansion of $G$ defined by $V\bar{G}=VG\cup\{\bar{R}\}$, $E\bar{G}=EG\cup\{\bar{e}\mid e\in o^{-1}(R)\}$ where $\bar{R}$ and $\bar{e}$ are elements not in $VG$ and $EG$ resp. . The functions $o$ and $t$ will remain the same on $EG$ and will be defined on each $\bar{e}$ as follows
            \[\forall e\in o^{-1}(R),\ o(\bar{e})=\bar{R} \text{ and } t(\bar{e})=t(e).\]
        Clearly, $\mathcal{T}(G,R)\cong \mathcal{T}(\bar{G},\bar{R})$ since when we identify $\bar{R}$ with $R$ and $\bar{e}$ with $e$, $\bar{G}$ becomes $G$. Additionally, since we can reach (with a path) every neighbour of $R$ from $\bar{R}$, we must be able to reach every vertex $\neq R$ from $\bar{R}$. This includes every predecessor of $R$ (which exist by assumption). So we can also reach $R$ from $\bar{R}$. This makes $\bar{R}$ a root of $\bar{G}$. Now since $\mathcal{T}(\bar{G},\bar{R})$ is cocompact and $t^{-1}(\bar{R})=\emptyset$, we have an actual labelling $(l,l_p)$ on $\bar{G}$. Now since $R$ and $\bar{R}$ share a neighbour ($o^{-1}(R)\neq \emptyset$ since there are no sinks) we must have $l(R)=l(\bar{R})$. However, by construction we have $|o^{-1}(R)|=|o^{-1}(\bar{R})|$ and thus    
            \[|M_{l(R)}|=|o^{-1}(R)|+1=|o^{-1}(\bar{R})|+1=|M_{l(\bar{R})}|+1=|M_{l(R)}|+1\]
        giving us the desired contradiction.
    \end{proof}
    Putting all of these results together we get the final form of \thmref{ch2-thm1}.
    \begin{theorem}
        For any rooted graph without sinks, $\mathcal{T}(G,R)$ is cocompact if and only if $G$ admits an actual labelling
    \end{theorem}

    To see that the ``no sinks" conditions cannot be removed note that the following graph
        
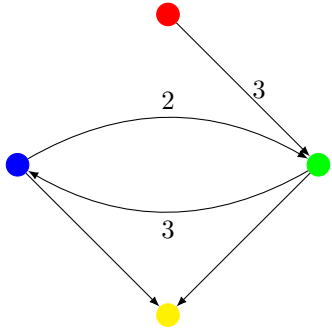
\begin{center}
            \begin{tikzpicture}
                \node[big red node]     (S) at (0,2)     {};
                \node[big blue node]    (x) at (-2,0)    {};
                \node[big green node]   (y) at (2,0)     {};
                \node[big yellow node]  (z) at (0,-2)    {};
                \path[-latex]
                        (x)     edge[bend left]                         node[above]             {2} (y)
                                edge                                    node[right]             {}  (z)
                        (y)     edge[bend left]                         node[below]             {3} (x)
                                edge                                    node[right]             {}  (z)
                        (S)     edge                                    node[right]             {3}  (y)

                          ;    
            \end{tikzpicture}
        \end{center}
    has a cocompact unfolding tree but does not admit an actual labelling. 

    \section{Trees almost isomorphic to cocompact trees}

    If we want to show when a tree is almost isomorphic to a label-regular tree, we note that it must be an unfolding tree of a graph. As is shown in \cite{gorazd2023classification}, the unfolding tree is almost isomorphic to some $\mathcal{T}(\llangle(G_1,\rho_1),\dots,(G_n,\rho_n)\rrangle,R)$, for disjoint $G_i$'s that are robust,non-redundant and connected. From the results of this paper, it follows that cocompact trees are of the form $\mathcal{T}(\llangle(H,\gamma)\rrangle,R)$, with $\llangle(H,\gamma)\rrangle$ non-redundant, admitting an actual labelling and $H$ robust (strongly connected except in the focal cycle case). So using the results from \cite{gorazd2023classification}, we can see that if the tree $\mathcal{T}(\llangle(G_1,\rho_1),\dots,(G_n,\rho_n)\rrangle,R)$ is almost isomorphic to a cocompact tree, we must have $n=1$.
    Now if a graph $\llangle G,\rho \rrangle$, has unfolding tree almost isomorphic to a cocompact unfolding trees,  we must have some function $\gamma$ s.t. $\llangle G,\gamma \rrangle$ admits an actual labelling. So $G$ must admit a pre-actual labelling. To check if $G$ admits a pre-actual labelling, we can do this using the following algorithm that refines the trivial labelling of the graph until it satisfies the second condition of \defref{ch2-def1}, while checking the third condition at every step.
    \begin{algorithm}[H]
        \caption{Constructing a pre-actual labelling of a graph}
        \begin{algorithmic}\label{ch2-alg1}
            \FOR{$v\in VG$}
                \STATE{$l_1(v)\gets 1$}
            \ENDFOR
            \STATE{$i\gets 1$}
            \REPEAT
                \STATE{$i\gets i+1$}
                \FOR{$v\in VG$}
                    \IF{$\exists l_p(v),\ \forall e\in t^{-1}(v),\ l_{i-1}(o(e))=l_p(v)$}
                        \STATE{$l_i(v)\gets l_{i-1}(t_m(o^{-1}(v)))\cup\{l_p(v)\}$}
                    \ELSE
                        \STATE{$G$ has no pre-actual labeling}
                    \ENDIF
                \ENDFOR
            \UNTIL{$\forall v,w\in VG,\ \big(l_{i-1}(v)=l_{i-1}(w)\big)\iff \big(l_i(v)=l_i(w))\big)$}
        \end{algorithmic}
    \end{algorithm}

     This algorithm gives us the coarsest pre-actual labelling $l$ on $G$ (if it exists). If we want to check whether a rooted graph $(G,R)$ has an actual labelling we first check with the above algorithm whether $G\setminus \{R\}$ has a pre-actual labelling. Then this pre-actual labelling can be expanded to the root making it an actual labelling iff we have a label $l_R$ s.t.  
     \begin{itemize}
         \item $\forall v \in t(o^{-1}(R)), l_p(v)=l_R$
         \item $M_{l_R}=t_m(o^{-1}(R))$.
     \end{itemize}
     Since we have constructed the coarsest pre-actual labelling, an actual labelling exists only if the constructed pre-actual labeling can be expanded. This answers the open question (1) in \cite{COURCELLE2025} in the finite sinkless case, I suspect it can be extended to the case with infinite edge valencies similarly. Graphs with sinks can be dealt with by removing all sinks and checking whether the resulting graph has an actual labeling.
     
     We see from \cite{gorazd2023classification} that any spider product of robust graphsm that is almost isomorphic to $\llangle G,\rho\rrangle$ is of the form $\llangle G,\gamma\rrangle$. So if we can find all the $\gamma$'s s.t. $\llangle G,\gamma\rrangle$ is actually labelled, we can check whether each one of them is almost isomorphic to $\llangle G,\rho\rrangle$, as in \cite{gorazd2023classification}, by checking wheather $\sum_{i\in\text{dom}{\gamma}}\gamma(i)=\sum_{i\in\text{dom}{\gamma}}\gamma(i)$ in the graph monoid.
     
     To find all these $\gamma$'s we note that the coarsest actual labelling on $\llangle G,\gamma\rrangle$ restricts to the coarsest pre-actual labelling on $G$. So if we find the coarsest pre-actual labelling $(l,l_p)$ using the above algorithm, we must find $\gamma$ s.t. $l$ extends to an actual labelling on $\llangle G,\gamma\rrangle$. 
     
     To see when a pre-actual labelling can be extended to an actual labelling, we note that in a non-redundant, actually labelled rooted graph $(H,S)$ then for any $v\in VH$ that shares a neighbour with $S$ we must also have $t(o^{-1}(S))\subseteq t(o^{-1}(v))$ (by non-redundancy) and $l_m(t_m(o^{-1}(S)))=M_{l(v)}$. Conversely, if $H\setminus\{S\}$ has a pre-actual labelling and there is a vertex $v$ with  $t(o^{-1}(S))\subseteq t(o^{-1}(v))$ and $l_m(t_m(o^{-1}(S)))=M_{l(v)}$ we can extend the labelling into an actual one by setting $l(S)=l(v)$.
     
     To check whether the pre-actual labelling on $G$ can be extended to an actual one on $\llangle G,\gamma\rrangle$ we just have to check whether $\gamma([n])\subseteq t(o^{-1}(v))$ and $l_m(\gamma_m(\text{dom}(\gamma))))=M_{l(v)}$, for some $v\in VG$. Note that since $G$ is non-redundant such a $\gamma$ is completely determined by $v$ (up to permutation which doesn't change the structure of the spider product), we will thus write $\gamma:=\gamma_v$ and conversely if there is some neighbour $w$ of $v$ with $l(w)=l_p(v)$, $v$ determines a $\gamma_v$. So if we want to check whether $\llangle G,\rho\rrangle$ has unfolding tree almost isomorphic to a cocompact tree, we just have to check whether it is almost isomorphic to the unfolding tree of some $\llangle G,\gamma_v\rrangle$. This is the case if
        \[\sum_{j\in [m]}\rho(j)=\sum_{i\in [n]}\gamma_v(i)=\sum_{e\in o^{-1}(v)} t(e)+v_p\]
    in the graph monoid, where $v_p\in t(o^{-1}(v))$ is the (unique) neighbour of $v$ s.t. $l(v_p)=l_p(v)$.
    
    Summing up, we can check whether $\mathcal{T}(H,S)$ is almost isomorphic to a cocompact tree, for any rooted graph $(H,S)$, as follows:
    \begin{itemize}
        \item Use \cite{gorazd2023classification}*{Lemma 10} to find a spider product of robust graphs
        \[\llangle(G'_1,\rho'_1,k'_1),\dots(G'_m,\rho'_m,k'_m)\rrangle\]
        with unfolding tree almost isomorphic to $\mathcal{T}(H,S)$.
        \item Factor over the coarsest non-edge-collapsing equivalence relation (that is trivial on the root) to get a non-redundant spider product of disjoint connected graphs $\llangle(G_1,\rho_1,k_1),\dots(G_n,\rho_n,k_n)\rrangle$, that has  unfolding tree almost isomorphic to $\mathcal{T}(H,S)$. If $n\neq 1$ then $\mathcal{T}(H,S)$ is not almost isomorphic to a cocompact tree. If $n=1$ write the spider product as $\llangle G,\rho\rrangle$.
        \item Use \algref{ch2-alg1} to determine whether $G$ has a pre-actual labelling. If not, $\mathcal{T}(H,S)$ is not almost isomorphic to a cocompact tree. If it admits a pre-actual labelling denote it by $(l_p,l)$.
        \item For any vertex $v\in VG$ that has an outgoing neighbour $v_p$ s.t. $l(v_p)=l_p(v)$, check whether
        \[\sum_{j\in [m]}\rho(j)=\sum_{e\in o^{-1}(v)} t(e)+v_p.\]
        If that is not the case for any $v$, $\mathcal{T}(H,S)$ is not almost isomorphic to a cocompact tree. If that is the case for some $v$, $\mathcal{T}(H,S)$ is almost isomorphic to $\mathcal{T}(\llangle G,\gamma_v\rrangle)$ which is a cocompact tree. 
    \end{itemize}
\section{Acknowledgements}
This paper came to be because of discussions with Alejandra Garrido, Waltraud Lederle and Colin Reid at the "Actions of totally disconnected locally compact groups on discrete structures" focus programme at the University of Münster. This paper has been adapted from a chapter of the authors PhD thesis, supervised by George Willis, Colin Reid and Stephan Tornier, reviewed by Collin Bleak and Roozbeh Hazrat, and supported by the ARC grant no. FL170100032. 
\bibliographystyle{plain}
 \bibliography{bibliography}
 
\end{document}